\documentclass[a4paper,12pt]{amsart}
\usepackage[utf8]{inputenc}
\usepackage[T1]{fontenc}
\usepackage[UKenglish]{babel}
\usepackage[a4paper,margin=28mm]{geometry}
\usepackage{verbatim}

\allowdisplaybreaks[4]
\usepackage{times}
\usepackage{dsfont,mathrsfs}
\usepackage{amsmath}
\usepackage{amsthm}
\usepackage{amssymb}
\usepackage{amsfonts}
\usepackage{latexsym}
\usepackage{booktabs}

\usepackage[colorlinks,linkcolor=black,pagebackref]{hyperref}
\usepackage{xcolor}

\usepackage{ulem}

\newtheorem{theorem}{Theorem}[section]
\newtheorem{lemma}[theorem]{Lemma}
\newtheorem{proposition}[theorem]{Proposition}

\theoremstyle{definition}

\newtheorem{remark}[theorem]{Remark}
\numberwithin{equation}{section}
\renewcommand{\labelenumi}{\roman{enumi})}
\renewcommand\theenumi\labelenumi
\renewcommand{\leq}{\leqslant}
\renewcommand{\le}{\leqslant}
\renewcommand{\geq}{\geqslant}
\renewcommand{\ge}{\geqslant}

\newcommand{\tl}{\tilde}
\newcommand{\Be}{\begin{equation}}
\newcommand{\Ees}{\end{equation*}}
\newcommand{\Bes}{\begin{equation*}}
\newcommand{\Ee}{\end{equation}}

\newcommand{\R}{\mathbb{R}}
\newcommand{\E}{\mathbb{E}}
\newcommand{\e}{\varepsilon}
\newcommand{\eps}{\epsilon}
\newcommand{\PP}{\mathbb{P}}
\newcommand{\N}{\mathbb{N}}
\newcommand{\mcl}{\mathcal}

\newcommand{\dif}{\mathrm{d}}

\usepackage{mathrsfs}
\usepackage{amsbsy}

\usepackage{marginnote}
\begin{document}
\title[Approximation the limiting diffusion of G/Ph/n+GI and related asymptotics]
{An approximation to the  invariant measure of the limiting diffusion of G/Ph/n+GI queues in the Halfin-Whitt regime and related asymptotics}

\author[X. H. Jin]{Xinghu Jin}
\address{Xinghu Jin: 1. Department of Mathematics,
Faculty of Science and Technology,
University of Macau,
Av. Padre Tom\'{a}s Pereira, Taipa
Macau, China; \ \ 2. UM Zhuhai Research Institute, Zhuhai, China.}
\email{yb77438@connect.um.edu.mo}

\author[G. Pang]{Guodong Pang}
\address{Guodong Pang: Department of Computational Applied Mathematics and Operations Research, 
George R. Brown College of Engineering,
Rice University,
Houston, TX 77005}
\email{gdpang@rice.edu}

\author[L. Xu]{Lihu Xu*}
\address{Lihu Xu: 1. Department of Mathematics,
Faculty of Science and Technology,
University of Macau,
Av. Padre Tom\'{a}s Pereira, Taipa
Macau, China; \ \ 2. UM Zhuhai Research Institute, Zhuhai, China.
Corresponding Author.}
\email{lihuxu@umac.mo}

\author[X. Xu]{Xin Xu}
\address{Xin Xu: 1. Department of Mathematics,
Faculty of Science and Technology,
University of Macau,
Av. Padre Tom\'{a}s Pereira, Taipa
Macau, China; \ \ 2. UM Zhuhai Research Institute, Zhuhai, China.}
\email{yb77439@umac.mo}


\begin{abstract}
In this paper, we develop a stochastic algorithm based on the Euler--Maruyama scheme to approximate the invariant measure of the limiting multidimensional diffusion of  $G/Ph/n+GI$ queues in the Halfin-Whitt regime. 
Specifically, we prove a non-asymptotic error bound between the invariant measures of the approximate model from the algorithm and the limiting diffusion.  
To establish the error bound, we employ the recently developed Stein's method for multi-dimensional diffusions, in which the regularity of Stein's equation developed by Gurvich (2014, 2022) plays a crucial role.

We further prove the central limit theorem (CLT) and the moderate deviation principle (MDP)  for the occupation measures of the limiting diffusion of $G/Ph/n+GI$ queues and its Euler-Maruyama scheme. In particular, the variances  in the CLT and MDP associated with the limiting diffusion are determined by Stein's equation and Malliavin calculus, in which properties of a mollified diffusion and an associated weighted occupation time play a crucial role.
\end{abstract}

\maketitle

\tableofcontents\thispagestyle{plain}

\section{Introduction} \label{Introduction}

A fundamentally important result in heavy-traffic queueing theory is the validity of diffusion approximations, that is, the interchange of limits property.
It provides an approximation of the steady state distribution of the queueing processes in a heavy-traffic regime by using the invariant measure of the limiting diffusion.
For instance, the interchange of limits results are proved for stochastic networks in \cite{BL,GZ,G14,YY16,YY18} and for many server queues \cite{AHP1,DDG1,GS,HAP22-OR,HAP22,S15}.
For some queueing models, the invariant measures of the limiting diffusions can be explicitly characterized \cite{DDG1,HW87}.
When this is impossible, numerical schemes are often drawn upon to compute the invariant measures, for example, computation of invariant measures of reflected Brownian motions in \cite{BCR1,DH92,DH1}. Our paper  is of similar flavor as Budhiraja et al. \cite{BCR1} where a Euler scheme approximation is developed for the constrained diffusions arising as scaling limits of stochastic networks.

In this paper, we focus on  $G/Ph/n+GI$ queues in the Halfin--Whitt regime. 
For many-server queues with exponential services, the limiting diffusions of the scaled queueing processes are one-dimensional with a piecewise-linear drift, whose steady state distributions have explicit expression as shown in \cite{BW,DDG1}. However, for  many-server queues with phase-type service time distributions, the limiting diffusions are multidimensional with a piecewise-linear drift, as shown in \cite{DHT1,PR}. Although the validity of diffusion approximations is proved for the $GI/Ph/n+M$ queues with renewal arrivals and exponential patience times in the Halfin--Whitt regime in Dai et al. \cite{DDG1}, the multi-dimensional limiting diffusion does not have an explicit invariant measure as shown in Dieker and Gao \cite{DG1}.
In fact, characterization of multi-dimensional piecewise diffusions has been left  as an open problem thus far in Browne et al. \cite{BW}. The objective in this paper is to provide an approximation for the invariant measure of the limiting diffusion of  $G/Ph/n+GI$ queues in the Halfin-Whitt regime. As a consequence, our result also provides an approximation for the steady-state of the diffusion-scaled queueing processes for  $GI/Ph/n+M$ queues in the Halfin--Whitt regime, given the justification of interchange of limits in Dai et al. \cite{DDG1}.  Moreover, 
 an approximation error bound is obtained for  the steady-state of the $M/Ph/n+M$ queues by applying the results in Gurvich \cite{Gur1} and Braverman and Dai \cite{BD1} (see Remark \ref{rem-prelimit-approx}).

\subsection{Summary of results and contributions}
The limiting diffusion  $(X_t)_{t\geq 0}$ satisfies the following stochastic differential equation (for short, SDE):
\begin{eqnarray}\label{hSDEg}
\dif X_{t} &=& g (X_{t} ) \dif t + \sigma \dif B_t
\end{eqnarray}
with $(B_t)_{t\geq 0}$ being a $d$-dimensional standard Brownian motion and
\begin{eqnarray*}
g(x) &=& -\beta p-Rx+(R-\alpha I)p({\rm e}' x)^+, \ \ \forall x\in \R^d\,.
\end{eqnarray*}
Here  $y^+=\max\{0,y\}$ for all $y\in \R$,   $\alpha> 0$ is the patience rate,  $\beta$ is the slack in the arrival rate relative to a critically loaded system,  $p \in \R^d$ is a vector of non-negative entries whose sum is equal to one, ${\rm e}=(1,1,\cdots,1)'$ with $'$ denoting the transpose,  $I$ is the identity matrix,
 \begin{eqnarray*}
R \ = \ (I-P') \textrm{diag}(v), \quad \frac{1}{\zeta} \ = \ {\rm e}' R^{-1}p, \quad
\gamma \ = \ \zeta R^{-1}p,
\end{eqnarray*}
where $v=(v_1,\cdots,v_d)$ with $v_k$ being the service rate in phase $k$,
 and $P$ be a sub-stochastic matrix describing the transitions between service phases such that $P_{ii}=0$ for $i=1,\cdots,d$, and $I-P$ being invertible (see Dai et al. \cite[Section 2.2]{DHT1}).
 Assume that the Ph phase distribution has mean $1$, that is, $\zeta=1$. It is easy to check ${\rm e}' \gamma=1$.
 $\sigma \sigma'$ has the following form:
\begin{eqnarray*}
\sigma \sigma' &=& \textrm{diag} (p)c_a^2 + H^{(0)} + \sum_{k=1}^d \gamma_k v_k H^{(k)} + (I- P')\textrm{diag}(v) \textrm{diag}(\gamma)(I-P),
\end{eqnarray*}
where $c_a^2>0$ is a constant, $\gamma=(\gamma_1,\cdots,\gamma_d)$, $H^{(k)}=(H^{(k)}_{ij})_{1\leq i, j \leq d}\in \R^{d\times d}$ with $H_{ii}^{(k)} = P_{ki}(1-P_{ki})$ and $H_{ij}^{(k)}=-P_{ki}P_{kj}$ for $j\neq i$ and $k =1, \dots, d$, and $H^{(0)}_{ii} = p_i(1-p_i)$ for $i=1,\dots,d$ and  $H^{(0)}_{ij} = -p_ip_j$ for $i\neq j$. Throughout this paper, we assume that there exists some constant $c>0$ such that  $\xi^{'}\sigma\sigma^{'}\xi \geq c\xi^{'}\xi$ for all  $\xi \in \R^d$. It is shown in Dieker and Gao \cite[Theorem 3]{DG1} that  the diffusion  $(X_t)_{t\geq 0}$ is exponentially ergodic and admits a unique invariant measure $\mu$. It is well known that for every $x\in \R^d$, SDE (\ref{hSDEg}) has a unique solution $(X^{x}_t)_{t\ge 0}$ starting from $x$ and the solution is nonexplosive from \eqref{e:AV} below and Meyn and Tweedie \cite[Theorem 2.1]{MT2}.
\begin{remark}
This diffusion $X_t$ is shown to be the limit of the diffusion-scaled queueing processes at all phases of the $G/Ph/n+GI$ queues in the Halfin--Whitt regime \cite{DHT1} assuming that the patience time distribution $F$ satisfies $F(0)=0$ and $\alpha :=\lim_{x\downarrow 0} F(x)/x<\infty$ and the service rate equals to 1.
In the covariance $\sigma \sigma'$,   $c_a^2$ captures the variability in the arrival process, specifically, when the arrival process is renewal, then $c_a^2$ is equal to the squared coefficient of variation of the interarrival times. The explicit form of $\sigma$ is not important for the result in Dieker and Gao \cite{DG1} and in our work.
\end{remark}

\smallskip

The {\bf Euler--Maruyama scheme} that we design to approximate the invariant measure $\mu$ of $(X_t)_{t\geq 0}$ reads as the following:
\begin{eqnarray}\label{e:XD}
\tl{X}_{k+1}^{\eta}
&=& \tl{X}_{k}^{\eta}+g(\tl{X}_{k}^{\eta}) \eta + \sqrt{\eta} \sigma \xi_{k+1},
\end{eqnarray}
where $k\in \mathbb{N}_0\triangleq \mathbb{N}\cup\{0\}$, $\tl{X}^{\eta}_{0}$ is the initial value and $\{\xi_k\}_{k\in \mathbb{N}}$ are the independent standard $d$-dimensional Gaussian random variables. $\tl{X}^{\eta,x}_k$ is the value of $k$-step with initial value $\tl{X}^{\eta}_0=x$.

Our first main result is the following theorem about this Euler--Maruyama (EM) scheme, which provides a non-asymptotic estimate for the error between the ergodic measures of the SDE and its EM scheme.

\begin{theorem}\label{thm:DDE}
$(\tl{X}_{k}^{\eta})_{k\in \mathbb{N}_0}$ defined by \eqref{e:XD} admits a unique invariant measure $\tl \mu_{\eta}$ and is exponentially ergodic. Moreover, the following two statements hold. 

(i) For any small enough $\varsigma \in (0,1)$, there exists some positive constant $C_\varsigma$, depending on $\varsigma$ but not on $\eta$, such that
\begin{eqnarray*}
d_{W}(\mu,\tl \mu_{\eta}) \ \leq \ C_\varsigma \eta^{\frac{1-\varsigma}{2}},
\end{eqnarray*}
where $d_W$ is the Wasserstein-1 distance, see \eqref{e:dW} below for the definition.

(ii)  For any (small) error $\delta>0$ and any small enough $\varsigma\in (0,1)$ in (i), taking $\eta=\delta^\frac{2}{1-\varsigma}$, we can run the EM algorithm $N:=O( \delta^{\frac{2}{\varsigma-1}} |\log \delta|)$ steps so that the law of $\tl{X}^\eta_N$, denoted as $\mcl L(\tl{X}^\eta_N)$, satisfies
$$d_W(\mcl L(\tl{X}^\eta_N),\mu) \ \leq \ \delta.$$
\end{theorem}

Our second set of main results are the central limit theorem (CLT), the moderate deviation principle (MDP) for the long term behavior of $(X_t)_{t\geq 0}$ and $(\tl{X}^{\eta}_k)_{k\in \mathbb{N}_0}$. For any $x \in \R^d$ and $T>0$, the empirical measure $\mcl E_T^x$ of $(X^x_t)_{t \ge 0}$ is defined by
 \begin{eqnarray*}
 \mcl E_T^x(A) \ = \ \frac 1T \int_0^T \delta_{X^x_s}(A) \dif s, \ \ \ \ \  A \in \mathcal{B}(\R^d),
 \end{eqnarray*}
where $\mathcal{B}(\R^d)$ is the collection of Borel sets on $\R^{d}$, $\delta_y(\cdot)$ is a delta measure, that is, $\delta_y(A)=1$ if $y \in A$ and $\delta_y(A)=0$ if $y \notin A$. It is easy to check that for any measurable function $h: \R^d \rightarrow \R$,
$$\mcl E_T^x(h) \ = \ \frac 1T \int_0^T h(X^x_s) \dif s.$$

For any $x \in \R^d$ and $n \in \mathbb{N}$, the empirical measure $ \mcl E_n^{\eta,x} $ of $(\tl{X}^{\eta,x}_k)_{k\in \mathbb{N}_0}$ is defined by
 \begin{eqnarray*}
 \mcl E_n^{\eta,x}(A) \ = \ \frac{1}{n} \sum_{k=1}^n \delta_{\tl{X}^{\eta,x}_k}(A), \ \ \ \ \  A \in \mathcal{B}(\R^d).
 \end{eqnarray*}
 It is easy to check that for any measurable function $h: \R^d \rightarrow \R$
$$\mcl E_n^{\eta,x}(h) \ = \ \frac{1}{n} \sum_{k=1}^n h(\tl{X}^{\eta,x}_k).$$
In order to state our theorems about CLT and MDP, for a given probability measure $\nu$, we define $L^2(\nu)$ as 
the Hilbert space induced by $\nu$ with inner product 
$$\langle f_1, f_2\rangle_{\nu}=\int f_1 f_2 d \nu \ \ \ \ \ {\rm for}  \ \ \ f_1, f_2 \in L^2(\nu).$$ 
For further use, we denote by $\mathcal{B}_b(\R^d,\R)$ the set of bounded measurable functions from $\R^d$ to $\R$, and denote by ${\rm Lip}(\R^d,\R)$ the set of globally Lipschitz functions from $\R^d$ to $\R$.

\begin{theorem} [CLT] \label{thm:CLT}
For any $h\in \mathcal{B}_b(\R^d,\R)$ and $x \in \R^d$,  $\sqrt{t} \left[\mcl E^x_t(h)-\mu(h)\right]$  converges weakly to Gaussian distribution  $\mathcal{N}(0,\mu(|\sigma^{\prime}  \nabla f|^2))$ as  $t \rightarrow \infty$, where $f$ is the solution to the Stein's equation \eqref{e:SE} below. Furthermore, $\mu(|\sigma' \nabla f|^2)\leq C\|h\|^2_{\infty}<\infty$ for some constant  $C>0$.
\end{theorem}

\begin{theorem}[MDP]\label{thm:MDP}
For any $h\in \mathcal{B}_b(\R^d,\R)$, $x \in \R^d$ and measurable set $A\subset \R$, one has
	\begin{eqnarray*}
	-\inf_{z \in A^{ {\rm o} } } \frac{ z^2 }{ 2 \mu(| \sigma^{\prime}  \nabla f |^2)  }
	\ &\leq& \
	\liminf_{t\to\infty}\frac{1}{a_t^2}\log \mathbb{P} \left( \frac{\sqrt{t}}{a_t}  \left[\mcl E^x_t(h)-\mu(h)\right]    \in A \right)  \\
	\ &\leq& \
	\limsup_{t\to\infty}\frac{1}{a_t^2}\log \mathbb{P} \left( \frac{\sqrt{t}}{a_t}  \left[\mcl E^x_t(h)-\mu(h)\right]  \in A \right)
\ \leq \ -\inf_{z \in \bar{A}} \frac{ z^2 }{ 2 \mu(| \sigma^{\prime}  \nabla f |^2)  },
	\end{eqnarray*}
	where $\bar{A}$ and $A^{ {\rm o} }$ are the closure and  interior of set $A$, respectively, and $a_t$ satisfies $a_t \to \infty$ and $\frac{a_t}{\sqrt{t}} \to 0$ as $t\to \infty$ and $f$ is the solution to the Stein's equation \eqref{e:SE} below.
\end{theorem}

\begin{theorem}[CLT and MDP]\label{thm:EMMDP}
(i). For any $h\in \mathcal{B}_b(\R^d,\R)$, $\sqrt{n} \big[ \mcl E_n^{\eta, x}(h)  -  \tl{\mu}_{\eta}(h) \big]$ converges weakly to Gaussian distribution $\mathcal{N}(0,\tl{ \mcl V}(h))$ as $n\to \infty$ with 
\begin{equation*}
	\tl{ \mcl V} (h)  \ \ = \ \   \langle f,  f \rangle_{\tl{\mu}_{\eta}}-\langle \tl{\mathcal{P}}_{\eta} f, \tl{\mathcal{P}}_{\eta} f \rangle_{\tl{\mu}_{\eta}},
	\end{equation*} 
where $f$ is the solution to the third Stein's equation \eqref{e:Stein-3} below and $\tl{\mcl P}_\eta f(x)=\E f(\tl X_1^{\eta,x})$. Moreover,  
	     \begin{eqnarray}\label{e:EMVh}
\tl{ \mcl V} (h) \ &=& \  \langle h-\tl{\mu}_{\eta}(h) , h-\tl{\mu}_{\eta}(h)   \rangle_{\tl{\mu}_{\eta}} + \sum_{k=1}^{\infty} \langle \tl{\mathcal{P}}^k_{\eta} h, h-\tl{\mu}_{\eta}(h) \rangle_{\tl{\mu}_{\eta}}.
	\end{eqnarray}	
(ii). For any $h\in \mathcal{B}_b(\R^d,\R)$, $x \in \R^d$ and measurable set $A\subset \R$, one has 
	\begin{eqnarray*}
	-\inf_{z \in A^{ {\rm o} } } \frac{ z^2 }{2 \tl{ \mcl V} (h)}
	\ &\leq& \
	\liminf_{n \to\infty}\frac{1}{a_n^2}\log \mathbb{P} \left( \frac{\sqrt{n}}{a_n}  \left[ \mcl E_n^{\eta,x}(h) - \tl{\mu}_{\eta}(h)\right]    \in A \right)  \\
	\ &\leq& \
	\limsup_{n \to\infty}\frac{1}{a_n^2}\log \mathbb{P} \left( \frac{\sqrt{n}}{a_n}  \left[ \mcl E_n^{\eta,x}(h) - \tl{\mu}_{\eta}(h)\right]  \in A \right)
\ \leq \ -\inf_{z \in \bar{A}} \frac{ z^2 }{ 2 \tl{ \mcl V} (h) },
	\end{eqnarray*}
	where $\bar{A}$ and $A^{ {\rm o} }$ are the closure and  interior of set $A$, respectively,  and $a_n$ satisfies $a_n \to \infty$ and $\frac{a_n}{\sqrt{n}} \to 0$ as $n\to \infty$.
	
\end{theorem}

\begin{remark}
We shall see below that Stein's equation will play an important role in proving Theorems \ref{thm:DDE}, \ref{thm:CLT} and \ref{thm:MDP}. The proof of Theorem \ref{thm:EMMDP} is a direct application of Jones \cite[Theorem 9]{jones2004markov} and Wu \cite[Theorem 2.1]{WLM1}, but we can see that a discrete Stein's equation \eqref{e:Stein-3} below plays an important role as well from the statement of the theorem. 
\end{remark}

\subsection{Related works}

Our paper is relevant to the following streams of works in the literature.

{\it (a) Steady state analysis of many-server queues.}
A few significant results have been obtained for understanding the steady-state of many-server queues, see, e.g., \cite{AR20,AHP1,AHPS1,APS1,BD1,BDF,GG13b,GG13a,Gur1,HAP22-OR,HAP22} and references therein.
The most relevant to us is the recent development using Stein's method to analyze the steady state of queueing processes via diffusion approximations.
Gurvich \cite{Gur1} provides a framework of analyzing the steady states via direct diffusion approximations (rather than diffusion limits) for a family of continuous-time exponentially ergodic Markov processes,
in particular, the gap between the steady-state moments of the diffusion models and those of the Markov processes is characterized. This result can be applied to Markovian many-server queueing systems.
Braverman et al.  \cite{BDF} introduced the Stein's method framework formally, proving Wasserstein and Kolmogorov distances between the steady-state distributions of the queueing processes and approximate diffusion models, and applied to the classical Erlang A and C models, where the bound is characterized by the system size.
Braverman  and Dai \cite{BD1} then extended this approach for the $M/Ph/n+M$ queues.
 Braverman et al. \cite{BDF20}  recently studied high order steady-state approximations of 1-dimensional Markov chains and applied to Erlang C models.

As discussed before, the invariant measure of the limiting diffusion of $G/Ph/n+GI$ queues in the Halfin--Whitt regime lacks an explicit expression and is difficult to compute directly. Dai and He \cite{DH1} proposed a numerical method to compute the invariant measure of the limiting diffusion  by solving the basic adjoint relationship, however, their work does not give theoretical error bounds. 
In this work, we provide a stochastic algorithm based on the EM scheme to compute the invariant measure for the limiting diffusion. More importantly,  our work characterizes the non-asymptotic error bound in terms of the step size in the algorithm. 

Although this approach uses discretization, in comparison with the discrete event simulation approach, the EM scheme is much more efficient computationally, especially when the number of servers $n$ is large in the queueing model.

It is worth noting that for the models with the one-dimensional limiting diffusions of linear or piecewise-linear drifts, since the invariant measure has an explicit density, when studying the steady-state approximation problem by Stein's method, one can explicitly solve Stein's equation and obtain the desired regularity properties easily. That is very similar to the one dimensional normal approximation case. It is well known that the generalization of Stein's method from one to multi-dimensional approximations is highly nontrivial \cite{CM08,RR-09}.
 Our problem is a multidimensional diffusion approximation. 

\smallskip 

{\it (b) Error estimates of EM schemes for diffusions.}
Let us recall the results concerning the error estimates between the ergodic measures of SDEs and their EM scheme.
For the ease of stating and comparing the results in the literatures below, we denote in this subsection by  $(Y_{t})_{t \ge 0}$ and $(\tl{Y}_{n})_{n \in \N_{0}}$ the stochastic processes associated to SDEs and their EM scheme respectively, and by $\pi_{sde}$ and $\pi_{em}$ their ergodic measures respectively.

There have been many results concerning the error estimates between the ergodic measures of $(Y_{t})_{t \ge 0}$ and $(\tl{Y}_{n})_{n \in \N_{0}}$, see for instance \cite{BBC1,BDMS1, BCR1, DM2, DM1, KT1, LP1, P1, P2, T3}, but most of them are
asymptotic type.  For asymptotic results, we recall those in the literatures \cite{BCR1,GHL1, P1, P2} whose settings are close to ours. An empirical measure $\Pi^{n}_{em}$ of $\pi_{em}$ for a class of SDEs driven by multiplicative L\'evy noises is considered in  \cite{GHL1, P1, P2}, and it is shown  that $\frac1{\sqrt{\Gamma_n}}(\Pi^{n}_{em}(f)-\pi_{sde}(f))$ converges to a normal distribution as $n \rightarrow \infty$ for $f$ in a certain high order differentiable function family ($\Gamma_n$ has the same order as $n$). 
A similar type CLT is obtained in Budhiraja et al. \cite{BCR1} for a reflected SDE driven arising in queueing systems. All these works need strong dissipation and high order differentiability conditions on the drift of SDEs, which do not hold for the limiting diffusions of our queueing systems.

Among the few non-asymptotic results, the works in \cite{BDMS1,DM2,DM1}, arising from the sampling of Langevin dynamics, are probably most close to ours. Their SDEs are gradient systems, i.e., the drift is the gradient of a potential $U$, thus analytical tools such as concentration inequalities are available. 
Under certain conditions on the drift, they prove non-asymptotic bounds for the total variation or Wasserstein-2 distance between $\pi_{sde}$ and $\pi_{em}$. Their analysis heavily depends on the gradient form of the drift, and is not easily seen to be extended to a non-gradient system. Our SDE is not a gradient system in that its drift $g(x)$ can not be represented as a gradient of a potential, and what is worse is that $g(x)$ is even not differentiable.

 There are some works (see \cite{BBC1,T3} and references therein) giving non-asymptotic results for the difference between the law of $\tl{Y}_n$ and $\pi_{sde}$ for large $n$. Most of these works need strong dissipation and high order differentiability assumptions on the drift of the SDE, and their estimates are in the form
$|\E h(\tl{Y}_n)-\pi_{sde}(h)|$ or $|\frac{1}n \sum_{i=1}^n h(\tl{Y}_i)-\pi_{sde}(h)|$ for $h$ in a certain high order differentiable function family, from which one usually cannot derive a bound between the law of $\tl{Y}_n$ and $\pi_{sde}$ in a Wasserstein distance.

\smallskip 

{\it (c) CLT and MDP with respect to ergodic measures.}  Dieker and Gao \cite{DG1} proved that SDE (\ref{hSDEg}) is exponentially ergodic with ergodic measure $\mu$, it implies that Birkhoff  ergodic theorem holds true for the empirical measure $\mcl{E}_T$ of the process $X_t$, i.e. $\lim_{T \rightarrow \infty} \mcl{E}_T=\mu$ a.s., see for instance Prato and Zabczyk \cite{da1996ergodicity}.
It is natural to consider the CLT and MDP with respect to $\mcl{E}_T$.  
Using the method in \cite{WXX1,WLM1}, we will establish the CLT and MDP for $\mcl{E}_T$, in which the related variances can be determined by a Stein's equation. Because $\mcl{E}_T$ is a random measure, it is natural to choose $\mathcal B_b(\R^d,\R)$ rather than ${\rm Lip}(\R^d,\R)$ as the test functions family, which makes the regularity results in Gurvich \cite[Theorem 4.1]{Gur1,Gur11C} not applicable. Alternatively, we apply Malliavin calculus to study the regularity of the Stein's equation. 
We also prove the CLT and MDP for the EM scheme in which the variance is determined by another Stein's equation.  There exist very few results for studying CLT and MDP of the EM scheme, see \cite{fukasawa2020efficient,lu2020central}. 

\subsection{Organization of the paper}

In the remainder of this section, we introduce notations which will be frequently used. Section \ref{sec-proof strategy} gives the proof of Theorem \ref{thm:DDE}, while Sections \ref{s:CLTMDP} and \ref{sec:EMCLTMDP} provide the proofs for the CLTs and MDPs with respect to the process $(X_t)_{t \ge 0}$ and the EM scheme $(\tl X^{\eta}_k)_{k \in \N_0}$ respectively. We prove in Appendix \ref{App:GeneralErgodicEM} the ergodicity of the EM scheme, and prove in Appendix \ref{sec:AAS} the propositions and Lemmas in Section \ref{sec:EMCLTMDP}.

\subsection{Notations}
Let $\R$ and $\mathbb{C}$ be real numbers and complex numbers respectively. The Euclidean metric is denoted by $| \cdot |$. For matrixes $A=(A_{ij})_{d\times d}$ and $B=(B_{ij})_{d\times d}$, denote $\langle A, B \rangle_{ {\rm HS} }=\sum_{i,j=1}^d A_{ij} B_{ij}$ and Hilbert Schmidt norm is $ \| A \|_{  {\rm HS} } = \sqrt{ \sum_{i,j=1}^d A_{ij}^2   } $ and operator norm is $\| A \|_{ {\rm op} } = \sup_{  |u|=1 } |Au|$.
We write a symmetric matrix $A>0\, (A<0)$ if $A$ is a positive (negative) definite matrix, and write $A\ge 0 \, (A \leq0)$ if $A$ is a positive (negative) semi-definite matrix. $\langle x, y \rangle$ means the inner product, that is, $\langle x, y \rangle = x^{\prime} y$ for $x,y\in \R^d$. $\otimes$ is the outer product, that is, for vector $u=(u_1,\cdots,u_d)$ and matrix  $A=(A_{ij})_{d\times d}$, then $(u \otimes A)_{ijk}=u_{i}A_{jk}$ for $1 \le i, j, k \le d$.

$\mathcal{C}^k(\R^d, \R_+)$ means $\R_+$-valued $k$-times continuous derivatives functions defined on $\R^d$ with $k\in \N$ and $\R_+=[0,\infty)$. $ \mathcal{C}_b(\R^d, \R)$ is $\R$-valued continuous bounded functions defined on $\R^d$. 
Denote $\| f \|_{\infty}= {\rm ess}\sup_{x\in \R^d} |f(x)|$ for $f\in\mathcal{B}_b(\R^d, \R )$. For $f\in \mathcal{C}^2(\R^d, \R)$, denote $\nabla f=(\partial_1 f, \partial_2 f, \cdots, \partial_d f) \in \R^d$ and $\nabla^2 f= (   \partial_{ij} f)_{1\leq i,j \leq d} \in \R^{d\times d}$ the gradient and Hessian matrix for function $f$. For $f\in \mathcal{C}^1(\R^d, \R)$ and $u, x\in \R^d$, the directional derivative $\nabla_{u} f (x)$ is defined by
\begin{eqnarray*}
\nabla_{u} f(x) &=& \lim_{\e_1 \to 0} \frac{f(x+\e_1 u) - f(x)}{\e_1}.
\end{eqnarray*}
We know $\nabla f (x) \in \R^d$ for each $x \in \R^{d}$ and $\nabla_{u} f (x) =\langle \nabla f(x), u \rangle$. For $f \in \mathcal{C}^2(\R, \R)$, $\dot{f}$ and $\ddot{f}$ are the first and second derivatives of function $f$, respectively. For any probability measure $\nu$, denote $\nu(f) = \int f(x) \nu(\dif x)$.

$B(y,r)$ means the ball in $\R^d$ with centre $y \in \R^d$ and radius $r>0$, that is, $B(y,r)=\{z\in \R^d:|z-y|\leq r\}$.

$\mathcal{N}(a,A)$ with $a\in \R^d$ and $A\in \R^{d\times d}$ denotes Gaussian distribution with mean $a$ and covariance matrix $A$.

A sequence of random variables $\{Y_n, n\geq 1 \}$ is said to converge weakly or converge in distribution to a limit $Y_{\infty}$, that is, $Y_n \Rightarrow Y_{\infty}$ if $\lim_{n \rightarrow \infty} \E f(Y_{n})=\E f(Y_{\infty})$ for all bounded continuous function  $f$. In addition,  $Y_n \stackrel{p}{\longrightarrow} Y_{\infty}$ means convergence in probability, namely, $\lim_{n\to \infty} \PP(|Y_n - Y_{\infty}|>\delta)=0$ for all $\delta \geq 0$. $Y_n \stackrel{L^p}{\longrightarrow} Y_{\infty}$ means the $L^p$ convergence, that is, $\lim_{n\to \infty}\E |Y_n-Y_{\infty}|^p=0$.

Denote $X^x_t$ the process $X_t$ given $X_0=x$.  Denote by $P_{t}(x,\cdot)$ the transition probability of $X_{t}$ given $X_{0}=x$. Then the associated Markov semigroup $(P_t)_{t\ge 0}$ is given by,  for all $x\in\R^d$ and $f\in\mathcal{B}_b(\R^d,\R )$
\begin{eqnarray*}
P_tf(x)
\ = \ \E f(X^x_t)
\  = \ \int_{\R^{d}} f(y) P_{t}(x, \dif y), \quad \forall t\ge0.
\end{eqnarray*}
The generator $\mathcal{A}$ of $(X_t)_{t\geq 0}$ is given by, for $y\in\R^d$,
\begin{eqnarray}\label{e:A}
\mathcal{A} f(y) \ = \ \langle \nabla f(y), g(y)\rangle+\frac{1}{2} \langle \sigma\sigma^{\prime}, \nabla^2 f(y) \rangle_{{\rm HS}}, \quad f \in \mathcal{D}(\mathcal{A}),
\end{eqnarray}
where $\mathcal{D}(\mathcal{A})$ is the domain of $\mathcal{A}$, whose exact form is determined by the function space where the semigroup $(P_{t})_{t \ge 0}$ is located.

Denote by $\tilde{\mathcal{P}}_{\eta}(x,\cdot)$ the one step transition probability for the Markov chain $(\tl{X}_k^{\eta})_{k\in \mathbb{N}_0}$ with $\tl{X}_0^{\eta}=x$, that is, for $f\in \mathcal{B}_b(\R^d,\R )$ and $x\in \R^d$, one has
\begin{eqnarray*}
\tilde{\mathcal{P}}_{\eta}f(x) \ = \  \int_{\R^d} f(y)\tl{\mathcal{P}}_{\eta}(x,\dif y),
\end{eqnarray*}
and denote $\tl{\mathcal{P}}_{\eta}^k=\tl{\mathcal{P}}_{\eta} \circ  \tl{\mathcal{P}}_{\eta}^{k-1}$ for integers $k\geq 2$.

We denote by $\E^{\mu}$ the conditional expectation given that $X_{0}$ has a distribution $\mu$. If $\mu=\delta_{x}$, we write $\E^{x}=\E^{\delta_{x}}$. $\E_{\mathbb{P}}$ and $\E_{\mathbb{Q}}$ mean expectations under probability spaces $\mathbb{P}$ and $\mathbb{Q}$, respectively.

Recall that the following measure distances. The Wasserstein-1 distance between two probability measures $\mu_1$ and $\mu_2$ is defined as  (see Hairer and Mattingly \cite[p. 2056]{HM1})
\begin{eqnarray}\label{e:dW}
d_{W}(\mu_1,\mu_2)
&=& \sup_{h \in {\rm Lip(1)}}\left\{\int h(x) \mu_1 (\dif x) - \int h(x) \mu_2 (\dif x) \right \} \nonumber   \\
&=& \sup_{h \in {\rm Lip_0(1)}} \left \{\int h(x) \mu_1 (\dif x) - \int h(x) \mu_2 (\dif x) \right \} \nonumber \\
&=& \sup_{h \in {\rm Lip(1)}} \left \{\int h(x) \mu_1 (\dif x) - \int h(x) \mu_2 (\dif x), \ \ |h(x)| \le |x| \right \},
\end{eqnarray}
where ${\rm Lip(1)}$ is the set of Lipschitz function with Lipschitz constant $1$, that is,  ${\rm Lip(1)}=\{ h: |h(x)-h(y)|\leq |x-y|$ for all $x,y \in \R^d \}$, and ${\rm Lip_0(1)}:=\{h \in {\rm Lip(1)}: h(0)=0\}$.

The total variation distance (see Hairer \cite[p. 57]{MH2}) between two measures $\mu_1$, $\mu_2$ is defined by
\begin{eqnarray*}
\|\mu_1-\mu_2\|_{\rm{TV}}
&=& \sup_{ \substack{h \in \mathcal{B}_b(\R^d, \R), \,  \| h \|_{\infty}\leq1}  }  \left \{ \int_{\R^d} h(x) \mu_1(\dif x) - \int_{\R^d} h(x) \mu_2(\dif x) \right \} .
\end{eqnarray*}
Let $V:\R^d \rightarrow \R_+$ be a measurable function, define a weighted supremum norm on measurable functions (see Hairer \cite[p. 57]{MH2}) by
\begin{eqnarray*}
\| \varphi \|_{V} &=& \sup_{x\in \R^d } \frac{ | \varphi(x) | }{ 1+ V(x) },
\end{eqnarray*}
as well as the dual norm of measures by
\begin{eqnarray*}
\|\mu_1-\mu_2 \|_{\rm{TV}, \rm{V}}
&=& \sup \left\{  \int \varphi(x) \mu_1(\dif x) -\int \varphi(x) \mu_2(\dif x) :  \| \varphi \|_V \leq 1\right\}.
\end{eqnarray*}

An alternative expression for the weighted total variation norm is given by
\begin{eqnarray}\label{normTVV}
\|\mu_1-\mu_2\|_{\rm{TV}, \rm{V}}
&=&  \int_{\R^d} (1+V(x)) |\mu_1-\mu_2 | (\dif x),
\end{eqnarray}
where $\mu_1-\mu_2$ is a signed measure and $|\mu_1-\mu_2|$ is the absolute value of $\mu_1-\mu_2$. Under $V\ge 0$, one has the relation $\|\mu_1-\mu_2\|_{\rm{TV}} \leq \|\mu_1-\mu_2\|_{\rm{TV}, \rm{V}} $. If $1+V(x) \ge 1+c |x|^2 \ge c'|x|$ for some constants $c, c'>0$, it follows from \eqref{e:dW} and \eqref{normTVV} that there exists some constant $C>0$ such that
\Be \label{e:dWandTV}
d_W(\mu_1,\mu_2) \ \leq \ C  \|\mu_1-\mu_2\|_{\rm{TV}, \rm{V}}.
\Ee

Let $P_t^*$ be the dual operator of $P_t$ for all $t\geq 0$, that is, for some measurable set $A$ and measure $\mu_1$, one has
\begin{eqnarray*}
(P_t^* \mu_1 )(A) \ = \ \int_{\R^d} P_t(x,A) \mu_1(\dif x).
\end{eqnarray*}

We use the letter $C$ to represent a positive constant, which may be different from line to line. Denote
\begin{eqnarray}
C_{\rm op} \ & = & \  \|R\|_{\rm op}+\|(R-\alpha I)p {\rm e}'\|_{\rm op}, \label{Cop}  \\
\tl{C}_{\rm op} \ & = & \ C_{\rm op}+\| \sigma\sigma^{\prime} \|_{\rm HS}+1+\|R-\alpha I\|_{\rm op}+|\beta|, \label{tlCop}  \\
C_m \ & = & \ 2m^2\tl{C}_{{\rm op}} {\rm \ for \ integers \ } m\geq 2. \label{Cm}
\end{eqnarray}

\section{Proof of Theorem \ref{thm:DDE} } \label{sec-proof strategy}

We give the proof of Theorem \ref{thm:DDE} by the help of Proposition \ref{p:GeneralErgodicEM} below and the Stein's method developed in Fang et al. \cite{FSX1}. Proposition \ref{p:GeneralErgodicEM} is on the existence of a unique invariant measure and the exponential ergodicity of the Markov chain $(\tl{X}_{k}^{\eta})_{k\in \mathbb{N}_0}$ in \eqref{e:XD}  in the EM scheme, while we use the Stein's method via solving a Stein's equation. Note that even if an SDE is ergodic, its EM scheme may blow up (see \cite{FG1,MSH1}), so a careful study of these ergodicity properties is necessary. Its proof will be given in Appendix \ref{App:GeneralErgodicEM}.

\begin{proposition} \label{p:GeneralErgodicEM}
$(\tl{X}_{k}^{\eta})_{k\in \mathbb{N}_0}$ in \eqref{e:XD} admits a unique invariant measure $\tl \mu_{\eta}$ and is exponentially ergodic. More precisely, for any $k\in \mathbb{N}_0$ and measure $\nu$ satisfying $\nu(|\cdot|^2)<\infty$, one has
\begin{eqnarray*}
d_W( (\tl{\mathcal{P}}_{\eta}^k)^* \nu, \tl{\mu}_{\eta})
\ &\leq& \ C\eta^{-1} e^{-c k\eta} , \\
\| (\tl{\mathcal{P}}_{\eta}^k)^* \nu- \tl{\mu}_{\eta} \|_{\rm TV}
\ &\leq& \ C\eta^{-1} e^{-c k\eta},
\end{eqnarray*}
where $C, c$ are positive constants independent of $k$ and $\eta$. Moreover, for integers $\ell \geq 2$, there exists some positive constant $C$ depending on $\ell$ but not on $\eta$ such that
\begin{eqnarray*}
\tl \mu_{\eta}(|\cdot|^{\ell}) \ \ \leq \ \ C, \ \ \ \ \ell \ge 2.
\end{eqnarray*}
\end{proposition}

\subsection{The first Stein's equation} In order to apply the Stein's method developed in Fang et al. \cite{FSX1}, we consider our first Stein's equation as follows: for a Lipschitz function $h: \R^d \to \R $ with $\| \nabla h \|_{\infty}<\infty$,
\begin{eqnarray}\label{e:PoiLip}
\mathcal{A} f(x) \ = \ h(x) - \mu(h),
\end{eqnarray}
where $\mathcal{A}$ is defined as in \eqref{e:A}, and $\mu$ is the invariant measure for the process $(X_t)_{t\geq 0}$ in \eqref{hSDEg} with semigroup $(P_{t})_{t\geq 0}$.  Without loss of generality, we assume that $h\in {\rm Lip}_0(1)$. Then we can get the regularity for the solution to Stein's equation \eqref{e:PoiLip} from {Gurvich \cite[Theorem 4.1]{Gur1,Gur11C} (see also  Braverman and Dai \cite[Lemma 1]{BD1}).} 
The regularity results of the equation \eqref{e:PoiLip} in the following lemma will be used in the proof of Theorem \ref{thm:DDE}.
\begin{lemma}\label{lem:Lipregf}
Let $h\in {\rm Lip}_{0}(1)$ and $f$ be the solution to \eqref{e:PoiLip}. 
There exists some positive constant $C$ such that for $1\leq i,j\leq d$, 
\begin{eqnarray}
|f(x)| \ &\leq& \ C(1+|x|), \nonumber  \\
|\partial_i f(x)| \ &\leq& \ C(1+|x|^2), \nonumber  \\
|\partial_{ij} f(x)| \ &\leq& \ C(1+|x|^3), \label{e:2f}
\end{eqnarray}
where $\nabla f=(\partial_1 f, \partial_2 f, \cdots, \partial_d f) \in \R^d$ and $\nabla^2 f= (   \partial_{ij} f )_{1\leq i,j \leq d} \in \R^{d\times d}$ are the gradient and Hessian matrix for $f$,  respectively.
For any small enough $\varsigma\in (0,1)$, there exists some positive constant $C_\varsigma$ depending on $\varsigma$ such that for $1\leq i,j\leq d$,
\begin{equation}
\sup_{y\in \R^d: |y-x|<1} \frac{ | \partial_{ij} f(y)-\partial_{ij} f(x)| }{|y-x|^{1-\varsigma}}  \  \leq \ C_\varsigma (1+|x|^4).  \label{e:3f}
\end{equation}
\end{lemma}

\subsection{{Proof of Theorem \ref{thm:DDE} }}
\begin{proof}[Proof of Theorem \ref{thm:DDE}] (i) By Proposition \ref{p:GeneralErgodicEM}, we know $(\tl{X}_{k}^{\eta})_{k\in \mathbb{N}_0}$ in \eqref{e:XD} admits a unique invariant measure $\tl{\mu}_{\eta}$. Let the initial value $\tl{X}^{\eta}_{0}$ be the invariant measure $\tl \mu_{\eta}$.
We know $(\tl{X}^{\eta}_{k})_{k \in \mathbb{N}_0}$ is a stationary Markov chain. Denote $W=\tilde{X}_{0}^{\eta}$, $W'=\tilde{X}_{1}^{\eta}$ and $\delta = W'-W$. It is easy to see that
\begin{eqnarray}  \label{e:EDelW}
\mathbb{E}[\delta|W] \ = \ g(W) \eta \text{ \ \ and \ \ }
\mathbb{E}[\delta \delta^{\prime}|W] \  =\ g(W) g^{\prime}(W) \eta^2 +\sigma \sigma^{\prime}\eta.
\end{eqnarray}

Let $f$ be the solution to Eq. \eqref{e:PoiLip} with $h\in {\rm Lip}_0(1)$. Since $W$ and $W'$ have the same distribution, we have
\begin{eqnarray}\label{hMC51e}
0&=& \mathbb{E}f(W')-\mathbb{E}f(W) \nonumber \\
&=& \mathbb{E}[ \langle \delta, \nabla f(W) \rangle  ]
+\mathbb{E} \int_0^1 \int_0^1 r \langle \delta \delta^{\prime}, \nabla^2 f (W+\tl{r} r \delta ) \rangle_{\textrm{HS}} \dif \tl{r} \dif r \nonumber \\
&=&\mathbb{E}[  \langle g(W),\nabla f(W) \rangle ]\eta+\mathbb{E} \int_0^1 \int_0^1 r \langle \delta \delta^{\prime}, \nabla^2 f(W+\tl{r} r \delta ) \rangle_{\textrm{HS}} \dif \tl{r} \dif r, \nonumber
\end{eqnarray}
where the last inequality holds from below
\begin{eqnarray}\label{hMC52e}
\mathbb{E}[ \langle \delta, \nabla f(W) \rangle  ]
\ = \ \mathbb{E}[ \langle \mathbb{E}(\delta|W),\nabla f(W)\rangle  ]
\ = \ \mathbb{E}[  \langle g(W),\nabla f(W) \rangle ]\eta. \nonumber
\end{eqnarray}
In addition, one has
\begin{eqnarray*}
&& \mathbb{E} \int_0^1 \int_0^1 r \langle \delta \delta^{\prime}, \nabla^2 f(W+\tl{r} r \delta ) \rangle_{\textrm{HS}} \dif \tl{r} \dif r \nonumber \\
&=& \frac{1}{2}\mathbb{E} \langle \delta \delta^{\prime}, \nabla^2 f(W) \rangle_{\textrm{HS}}+
\mathbb{E} \int_0^1 \int_0^1 r \langle \delta \delta^{\prime}, \nabla^2 f(W+\tl{r} r \delta )-\nabla^2 f(W) \rangle_{\textrm{HS}} \dif \tl{r} \dif r \nonumber \\
&=& \frac{\eta}{2} \mathbb{E} [ \langle \sigma \sigma^{\prime}, \nabla^2 f(W) \rangle_{\textrm{HS}}  ] +\frac{\eta^2}{2} \mathbb{E} [  \langle g(W) g^{\prime}(W), \nabla^2 f(W) \rangle_{\textrm{HS}} ] \nonumber \\
&\ & \ \ \ \ \ \ \ \ \ \ \ \ \ \ \ \ \ \ \ +\mathbb{E} \int_0^1 \int_0^1 r \langle \delta \delta^{\prime}, \nabla^2 f (W+\tl{r} r \delta )-\nabla^2 f(W) \rangle_{\textrm{HS}} \dif \tl{r} \dif r,
\end{eqnarray*}
where the second equality is by the relation $\mathbb{E}[\langle \delta \delta^{\prime}, \nabla^2 f(W) \rangle_{\textrm{HS}}]=\mathbb{E}[\langle \E[\delta \delta^{\prime}|W], \nabla^2 f(W) \rangle_{\textrm{HS}}]$ and \eqref{e:EDelW}.
Collecting the previous relations, we obtain
\begin{eqnarray}\label{e:Afh}
\mathbb{E}[\mathcal{A}f(W) ] \ = \ \frac \eta 2 {\rm I}+\frac{1}{\eta} {\rm II},
\end{eqnarray}
where
\begin{align*}
{\rm I}  &\ = \ -\mathbb{E} [ \langle g(W) g^{\prime}(W), \nabla^2 f(W) \rangle_{\textrm{HS}} ], \\
{\rm II} &\ = \ -\mathbb{E} \int_0^1 \int_0^1 r \langle \delta \delta^{\prime}, \nabla^2 f(W+\tl{r} r \delta )-\nabla^2 f(W) \rangle_{\textrm{HS}} \dif\tl{r} \dif r.
\end{align*}

From the estimate $|g(x)| \leq \tl{C}_{\rm op}(1+|x|)$ for all $x\in \R^d$ with $\tl{C}_{{\rm op}}$ in \eqref{tlCop} and using  inequality \eqref{e:2f} in Lemma \ref{lem:Lipregf}, one has
\begin{eqnarray}\label{e:Afh-1}
|{\rm I}|
&=&\mathbb{E}|\langle g(W) g^{\prime}(W), \nabla^2 f(W) \rangle_{\textrm{HS}} |
\ \leq \ C(1+\E|W|^5)
\ \leq \ C,
\end{eqnarray}
where the last inequality is by Proposition \ref{p:GeneralErgodicEM}.

We claim that for any small enough $\varsigma \in (0,1)$, there exists some positive constant $C_\varsigma$, depending on $\varsigma$ but not on $\eta$ such that
\begin{eqnarray}\label{e:claim-1}
\mathbb{E} \left| \int_0^1 \int_0^1 r \langle \delta \delta^{\prime}, \nabla^2 f(W+\tl{r} r \delta )-\nabla^2 f(W) \rangle_{\textrm{HS}} 1_{\{|\delta|< 1\} } \dif\tl{r} \dif r \right|
\ &\leq& \ C_{\varsigma} \eta^{\frac{3-\varsigma}{2}},
\end{eqnarray}
and there exists some positive constant $C$, independent of $\eta$ such that
\begin{eqnarray}\label{e:claim-2}
\mathbb{E} \left| \int_0^1 \int_0^1 r \langle \delta \delta^{\prime}, \nabla^2 f(W+\tl{r} r \delta )-\nabla^2 f(W) \rangle_{\textrm{HS}} 1_{ \{ |\delta|\geq 1 \} } \dif\tl{r} \dif r \right|
\ &\leq&\  C\eta^{\frac{3}{2}}\,.
\end{eqnarray}

Combining \eqref{e:claim-1} and \eqref{e:claim-2}, we know for any small enough $\varsigma \in (0,1)$, there exists some positive constant $C_\varsigma$, depending on $\varsigma$ but not on $\eta$ such that
\begin{eqnarray}\label{e:Afh-2}
|{\rm II}| \ \leq \  C_{\varsigma} \eta^{\frac{3-\varsigma}{2}}.
\end{eqnarray}

Combining \eqref{e:Afh}, \eqref{e:Afh-1}, \eqref{e:Afh-2} and Eq. \eqref{e:PoiLip}, we know for any small enough $\varsigma \in (0,1)$, there exists some positive constant $C_\varsigma$, depending on $\varsigma$ but not on $\eta$ such that
\begin{eqnarray*}
d_{W}(\tl \mu_{\eta},\mu)
&=& \sup_{h\in {\rm Lip_0(1)}} |\E h(W)-\mu(h)|
\ = \ \sup_{h\in {\rm Lip_0(1)} } | \mathbb{E}\mathcal{A}f(W)|
\ \leq \  C_{\varsigma}\eta^{\frac{1-\varsigma}{2}}.
\end{eqnarray*}

It remains to show that inequalities \eqref{e:claim-1} and \eqref{e:claim-2} hold. It follows from \eqref{e:3f} in Lemma \ref{lem:Lipregf} and Proposition \ref{p:GeneralErgodicEM} that  for any small enough $\varsigma \in (0,1)$, there exists some positive constant $C_\varsigma$, depending on $\varsigma$ but not on $\eta$ such that
\begin{eqnarray*}
&& \mathbb{E} \left| \int_0^1\int_0^1 r \langle \delta \delta^{\prime}, \nabla^2 f(W+\tl{r} r \delta )-\nabla^2 f(W) \rangle_{\textrm{HS}} 1_{\{|\delta|< 1\} } \dif\tl{r} \dif r \right|  \\
&\leq& C_\varsigma \E[  |\delta|^{3-\varsigma} (1+|W|^4) 1_{\{|\delta|< 1\} } ] \\
&\leq& C_\varsigma \E[ |\eta g(W) + \sigma \eta^{\frac{1}{2}} \xi_1|^{3-\varsigma} (1+|W|^4) ]  \\
&\leq& C_\varsigma \eta^{\frac{3-\varsigma}{2}}  \E[1+|W|^7]  \\
&\leq& C_\varsigma \eta^{\frac{3-\varsigma}{2}},
\end{eqnarray*}
where the third inequality holds from small $\eta<1$. Thus, inequality \eqref{e:claim-1} holds.

It follows from \eqref{e:2f}  in Lemma \ref{lem:Lipregf} and Proposition \ref{p:GeneralErgodicEM} that for small $\eta<1$, there exists some positive constant $C$, independent of $\eta$ such that
\begin{eqnarray*}
&& \mathbb{E} \left| \int_0^1 \int_0^1 r \langle \delta \delta^{\prime}, \nabla^2 f(W+\tl{r} r \delta )-\nabla^2 f(W) \rangle_{\textrm{HS}} 1_{ \{ |\delta|\geq 1 \} } \dif\tl{r} \dif r \right|  \\
&\leq& C \mathbb{E} \int_0^1 \int_0^1 r |\delta|^2 |\nabla^2 f(W+\tl{r} r \delta )| 1_{ \{ |\delta|\geq 1 \} } \dif\tl{r} \dif r  \\
& & \quad + C \mathbb{E} \int_0^1 \int_0^1 r  |\delta|^2 |\nabla^2 f(W)| 1_{ \{ |\delta|\geq 1 \} } \dif\tl{r} \dif r  \\
&\leq& C\E[ |\delta|^2 (1+|W|^3+|\delta|^3) 1_{ \{ |\delta|\geq 1 \}  }  ]  \\
&\leq& C(\E[ |\delta|^4])^{\frac{1}{2}}  \PP^{\frac{1}{2}} ( |\delta|\geq 1 )
+C (\E[|W|^6])^{\frac{1}{2}} ( \E[ |\delta|^8] )^{\frac{1}{4}} \PP^{\frac{1}{4}}( |\delta|\geq 1 )
+C\E[ |\delta|^5  ]  \\
&\leq& C\eta^{\frac{3}{2}},
\end{eqnarray*}
where the last inequality holds from Chebyshev's inequality and Proposition \ref{p:GeneralErgodicEM}, that is,
\begin{eqnarray*}
\PP( |\delta|\geq 1)
\ \leq \ \E|\delta|^k
\ \leq \ C\eta^{\frac{k}{2}} \E[1+|W|^k]
\ \leq \ C \eta^{\frac{k}{2}}
\end{eqnarray*}
for any integers $k\geq 1$. Thus, inequality \eqref{e:claim-2} holds.

(ii) By triangle inequality and using Proposition \ref{p:GeneralErgodicEM}, for any small enough $\varsigma \in (0,1)$, there exists some positive constant $C_\varsigma$, depending on $\varsigma$ but not on $\eta$ such that
\begin{eqnarray*}
d_W( \mathcal{L}(\tl{X}_N^{\eta}),\mu)
&\leq& d_W( \mathcal{L}(\tl{X}_N^{\eta}),\tl{\mu}_{\eta})
+d_W(\tl{\mu}_{\eta},\mu)
\ \leq \  C\eta^{-1} e^{-c N\eta} + C_\varsigma \eta^{\frac{1-\varsigma}{2}}.
\end{eqnarray*}
Taking $\eta=\delta^{\frac{2}{1-\varsigma}}$ and $N:=O(\delta^{ \frac{2}{\varsigma-1} } |\log \delta|)$, one has $\eta^{-1} e^{-cN\eta} \leq \delta$, it implies that
\begin{eqnarray*}
d_W( \mathcal{L}(\tl{X}_N^{\eta}),\mu)
\ \leq \ \delta.
\end{eqnarray*}
The proof is complete.
\end{proof}

\begin{remark} \label{rem-prelimit-approx}
Recall the  steady-state approximation of the $M/Ph/n+M$ model was studied in Braverman and Dai \cite{BD1}. The process $X_t$ is the limiting diffusion of the diffusion-scaled process $\hat{X}^n(t) = \frac{X^n(t) - n\gamma}{\sqrt{n}}$. The process $\hat{X}^n(t)$ admits a unique ergodic measure $\hat \mu^n$ by  Gurvich \cite{Gur1}. The result in Braverman and Dai \cite{BD1}, implies that for any small enough $\varsigma \in (0,1)$, there exists some positive constant $C_\varsigma$, depending on $\varsigma$ but not on $n$ such that
\begin{equation}  \label{e:BrDa-0}
d_W(\hat \mu^n, \mu) \ \leq  \   C_\varsigma n^{\frac{\varsigma-1}{2}}.
\end{equation}
We apply the EM scheme developed in this paper to provide an approximation for the steady-state of $\hat{X}^n(t)$. 
By Theorem \ref{thm:DDE} (i) and \eqref{e:BrDa-0}, there exists some positive constant $C_\varsigma$, depending on $\varsigma$ but not on $n$ such that
\begin{eqnarray*}
d_W(\hat \mu^n, \tilde{\mu}_\eta) \ \leq \  C_\varsigma (n^{\frac{\varsigma-1}{2}} + \eta^{\frac{1-\varsigma}{2}})  \ \leq \   C_\varsigma n^{\frac{\varsigma-1}{2}},
\end{eqnarray*}
as $\eta$ is sufficiently small (say $\eta=n^{-1}$). Moreover, by Theorem \ref{thm:DDE} (ii), we can only run the EM scheme $N=O(n \log n)$ steps and obtain  that there exists some positive constant $C_\varsigma$, depending on $\varsigma$ but not on $n$ such that
\begin{eqnarray*}
d_W( \mathcal{L}(\tl{X}_N^{\eta}),\hat \mu^n) \ \leq \  C_\varsigma n^{\frac{\varsigma-1}{2}}.
\end{eqnarray*}
We expect a similar result holds for the general $GI/Ph/n+GI$ queues, if the result in \eqref{e:BrDa-0} can be established for non-Markovian queues. 
\end{remark}

\section{Proofs of Theorems \ref{thm:CLT} and \ref{thm:MDP}} \label{s:CLTMDP}
In order to determine the variance of CLT and MDP in Theorems \ref{thm:CLT} and \ref{thm:MDP}, we need to consider a new Stein's equation, whose regularity need to be studied by Malliavin calculus. All the proofs for the lemmas and propositions in this section are postponed to Appendix \ref{sec:AAS}.
\subsection{The second Stein's equation}
We consider the following Stein's equation: for $h\in \mathcal{B}_b(\R^d,\R)$,
\begin{eqnarray}\label{e:SE}
\mathcal{A} f(x) \ = \ h(x) - \mu(h),
\end{eqnarray}
where $\mathcal{A}$ is defined in \eqref{e:A}, and $\mu$ is the invariant measure for the process $(X_t)_{t\geq 0}$ in \eqref{hSDEg} with semigroup $(P_{t})_{t\geq 0}$.

The following proposition plays a crucial role in the proof Theorems \ref{thm:CLT} and \ref{thm:MDP}.
\begin{proposition}\label{lem:regf}
Let $h\in \mathcal{B}_b(\R^d, \R)$ and $f$ be the solution to the Stein's equation \eqref{e:SE}. 
There exists some positive constant $C$ such that
	\begin{eqnarray*}
|f(x)| \ &\leq& \ C\|h\|_\infty(1+|x|^2), \\
|\nabla f(x)| \ &\leq& \ C \|h\|_{\infty}(1+ |x|^2).
	\end{eqnarray*}
\end{proposition}

Let us sketch the strategy for the proof of Proposition \ref{lem:regf} as the following. Unlike Stein's equation \eqref{e:PoiLip} above, the function  $h$ on the right hand side of the equation \eqref{e:SE} belongs to $\mathcal B_b(\R^d,\R)$, whereby 
the regularity result in \cite[Theorem 4.1]{Gur1,Gur11C} cannot be applied to prove Proposition \ref{lem:regf}.  Alternatively, we first prove the following two representations for 
the solution $f$:
\begin{lemma}\label{prop:ST}
For any function $h\in \mathcal{B}_b(\R^d,\R)$, the following statements hold. 

(i) A solution to \eqref{e:SE} is given by
\begin{eqnarray}\label{e:SE1}
f(x) \ = \ - \int_0^{\infty} P_t[   h(x) - \mu(h) ] \dif t.
\end{eqnarray}

(ii) The solution to \eqref{e:SE} is also given by
\begin{eqnarray*} 
f(x) \ = \  \int_0^{\infty} e^{-\lambda t} P_t[  \lambda f(x) - h(x) + \mu(h) ] \dif t,  \quad  \forall  \lambda>0.
\end{eqnarray*}
\end{lemma}
Thanks to Lemma \ref{prop:ST}, to bound $\nabla f$ we only need to bound $\nabla P_t[   \lambda f(x) - h(x)  + \mu(h)]$, whereby it is natural to see that the powerful Bismut-Elworthy-Li formula can be applied. Due to the non-differentiability of $g$, we first mollify the drift $g(x)$ by $\e$, then apply Malliavin calculus to the consequent mollified SDE \eqref{hSDEgApp} below, and finally let $\e \rightarrow 0$ to obtain the following Bismut-Elworthy-Li type formula:  
\begin{lemma}\label{hLef2}
Let $\psi \in \mathcal{C}^1(\mathbb{R}^d,\mathbb{R})$ be such that $\| \nabla   \psi \|_{\infty} < \infty$. For every $t>0$, $x, u \in \mathbb{R}^d$, we have
\begin{eqnarray*}
\nabla_{u} \mathbb{E} [\psi (X^{x}_t)]  \ = \  \mathbb{E}[\psi(X^{x}_t) \mathcal{I}_{u}^x(t)],
\end{eqnarray*}
where $\mathcal{I}_{u}^x(t)$ is defined as the second relation in \eqref{e:Iuxt} below. 
\end{lemma}
The aforementioned strategy for the proof of Lemma \ref{hLef2} is detailed in Section \ref{sub:MolDif} below. We will see that the weighted occupation time plays a crucial role.   

Combining Lemmas \ref{prop:ST} and \ref{hLef2}, we prove Proposition \ref{lem:regf} with the details addressed in Appendix \ref{sec:AAS}. All the proofs for the lemmas and propositions in this section will be given in Appendix \ref{sec:AAS}.

\subsection{A mollified diffusion}  \label{sub:MolDif}
 We shall use the Malliavin calculus to get the regularity in Lemma \ref{lem:regf} for $h\in \mathcal{B}_b(\R^d,\R)$. Since $g(x)$ is not differentiable, we consider an approximation of SDE \eqref{hSDEg} by mollifying the drift:
\begin{eqnarray}\label{hSDEgApp}
\dif X^{\e}_t  \ = \ g_\e (X^{\e}_t) \dif t +\sigma \dif B_t,
\end{eqnarray}
where
\begin{eqnarray*}
g_{\e} (x) \ = \ -\beta p -Rx+\rho_{\e}({\rm e}^{\prime}  x) (R-\alpha I)p,
\end{eqnarray*}
and $\rho_{\e}$ is defined as below: for $0<\e<1$,
$$\rho_{\e}(y) \ = \  \begin{cases}
0, & \quad y<-\e, \\
y, & \quad y>\e, \\
\frac{3\e}{16} - \frac{ 1 }{ 16 \e^3} y^4 + \frac{ 3 }{8 \e} y^2 + \frac{1}{2} y, & \quad |y| \le \e. \\
\end{cases}$$
It is easy to check that $\|\dot{\rho}_{\e}\|_{\infty} \leq 1$ and that
\begin{eqnarray}
g_{\e} \ & \in & \ \mathcal{C}^2(\mathbb{R}^d,\mathbb{R}^d),  \nonumber  \\
\nabla g_{\e}(x) \ & = & \ -R +\dot{\rho}_{\e} ({\rm e}^{\prime}  x) (R-\alpha I)p {\rm e}^{\prime} , \label{e:Nge} \\
\nabla^2 g_{\e}(x) \ & =& \ \ddot{\rho}_{\e}({\rm e}^{\prime}  x){\rm e}\otimes(R-\alpha I)  p{\rm e}^{\prime} . \nonumber
\end{eqnarray}
Moreover, we can see that $ \lim\limits_{\e \rightarrow 0}g_{\e}(x)=g(x)$ for all $x\in \R^d$ and
$$\lim_{\e \rightarrow 0}\nabla g_{\e}(x) \ = \ \begin{cases}-R +(R-\alpha I)p {\rm e}^{\prime} , & \quad {\rm e}^{\prime} x>0,
\\
-R, & \quad {\rm e}^{\prime}  x<0,  \\
-R + \frac{1}{2}(R-\alpha I)p{\rm e}^{\prime}  , & \quad {\rm e}^{\prime} x=0.
\end{cases}$$
\subsubsection{Moment estimates and Jacobi flows}
\begin{lemma}\label{lem:XXem}
For all $x \in \R^{d}$, $t\geq 0$ and intergers $m \geq 2$, we have
\begin{eqnarray}  \label{e:XXem}
\E |X^{\e,x}_t|^{m}, \E|X^{x}_{t}|^{m} \ &\leq& \ e^{C_m t}(|x|^{m}+1),
\end{eqnarray}
where $C_m$ is in \eqref{Cm}.
Moreover, we have as $\e \to 0$,
\begin{eqnarray}
\E|X^{\e,x}_{t}-X^{x}_{t}|^m \rightarrow 0,  \qquad  t \geq 0. \label{e:XeCon-1}
\end{eqnarray}
\end{lemma}

We consider the derivative of $X^{\e,x}_t$ with respect to initial value $x$, which is called the Jacobi flow. Let $u \in \mathbb{R}^d$ and the Jacobi flow $\nabla_{u} X^{\e,x}_{t}$ along the direction $u$ be defined as
\begin{eqnarray*}
\nabla_{u} X^{\e,x}_{t} &=& \lim_{\e_1 \rightarrow 0} \frac{ X^{\e,x+\e_1 u}_t - X^{\e,x}_t  }{\e_1}, \quad t\geq 0.
\end{eqnarray*}
The above limit exists and satisfies
\begin{eqnarray*}
\frac{\dif}{\dif t}\nabla_{u} X^{\e,x}_{t} &=&\nabla g_{\e} (X^{\e,x}_t) \nabla_{u} X^{\e,x}_{t}, \quad \nabla_u X_{0}^{\e,x}=u.
\end{eqnarray*}

Define
\begin{eqnarray*}
J_{s,t}^{\e,x}:&=&\exp \left( \int_s^t \nabla g_{\e} (X_{r}^{\e,x}) \dif r \right), \ \ \ \ \ \ 0 \le s \le t<\infty.
\end{eqnarray*}
It is called the Jacobian between $s$ and $t$. For notational simplicity, denote $J^{\e,x}_{t}=J^{\e,x}_{0,t}$. Then we have
\begin{eqnarray}\label{hJaF1}
\nabla_{u} X^{\e,x}_{t}   &=&  J_{t}^{\e,x} u .
\end{eqnarray}

Define
$$\nabla \widetilde{g}(x):\ = \ -R +1_{\{{\rm e}^{\prime} x>0\}} (R-\alpha I)p {\rm e}^{\prime}.$$
It is easy to see that $\lim_{\e \rightarrow 0} \nabla g_{\e}(x)=\nabla \widetilde{g}(x)$ for all ${\rm e}^{\prime} x \ne 0$.
Because $g(x)$ is not differentiable for ${\rm e}^{\prime}  x = 0$, it is necessary for us to define the above $\widetilde{g}(x)$ which takes the same value as $\nabla g(x)$ for ${\rm e}^{\prime} x \ne 0$ and has a definition on ${\rm e}^{\prime}  x=0$. Define
$$J^{x}_{s,t}:\ = \ \exp\left(\int_{s}^{t} \nabla \widetilde{g} (X^{x}_{r}) \dif r\right), \ \ \ \ \ \ x \in \R^{d}, \  \ 0 \le s \le t<\infty.$$
We also denote $J^{x}_{t}=J^{x}_{0,t}$.
Then we have the following lemma.
\begin{lemma}  \label{l:XeCon}
For any $x \in \R^{d}$, as $\e \rightarrow 0$, the following holds: 
\begin{eqnarray*}
 \|J_{s,t}^{\e,x}-J^{x}_{s,t}\|_{{\rm op}}\ {\longrightarrow} \ 0, \ \ \ \ \ \ \ & & 0 \le s \le t<\infty, \ \ \ \ \ { \rm a.s. }
\end{eqnarray*}
\end{lemma}

Now we give estimates for  $\|J^{\e,x}_{s,t}\|_{ {\rm op} }$ and $\|J^{x}_{s,t}\|_{ {\rm op} }$. By \eqref{e:Nge}, we can easily see that
\begin{eqnarray}  \label{e:Ngeop}
\|\nabla g_{\e}(x)\|_{ {\rm op} } \ \le \ \|R\|_{ {\rm op} }+\|(R-\alpha I) p {\rm e}^{\prime} \|_{ {\rm op} } \ = \ C_{{\rm op}},
\end{eqnarray}
from which we obtain
\begin{eqnarray*}
	\|J^{\e,x}_{s,t}\|_{ {\rm op} }
	&  \le  & \exp\left(\int_s^{t} \|\nabla g_{\e}(X^{\e,x}_{r})\|_{ {\rm op} } \dif r\right) \ \le \ e^{C_{{\rm op}}(t-s)}.
\end{eqnarray*}
So for all $0 \le s \le t<\infty$, we have
\begin{eqnarray}\label{e:JJe}
\|J^{x}_{s,t}\|_{ {\rm op} },\|J^{\e,x}_{s,t}\|_{ {\rm op} } \  \le \ e^{C_{{\rm op}}(t-s)},
\end{eqnarray}
where the bound of $\|J^{x}_{s,t}\|_{ {\rm op} }$ comes from the same argument since the bound in \eqref{e:Ngeop} also holds for $\nabla \widetilde{g}(x)$. Observe that the above estimates immediately implies that for $u\in\R^d$,
\begin{eqnarray*}
|\nabla_u X_t^x|,\,|\nabla_u X_t^{\e,x}| \ \leq  \  e^{C_{{\rm op}}t}|u|.
\end{eqnarray*}

\subsubsection{Bismut's formula of Malliavin calculus for the mollified diffusion} 
Let $v\in L^2_{\rm{loc}} ([0,\infty) \times (\Omega, \mathcal{F}, \mathbb{P}), \mathbb{R}^d)$, that is, $\mathbb{E} \int_0^t |v(s)|^2 \dif s < \infty$ for all $t >0$. Assume that $v$ is adapted to the filtration  $ (\mathcal{F}_t)_{t\geq 0}$ with $\mathcal{F}_t= \sigma(B_s: 0\leq s \leq t)$, that is, $v(t)$ is  $\mathcal{F}_t$ measurable for $t\geq 0$. Define
\begin{eqnarray}\label{hMCV}
\mathbb{V}(t) \  = \  \int_0^t v(s) \dif s, \quad t\geq 0.
\end{eqnarray}
For $t>0$, let $F_t: \mathcal{C}([0,t],\mathbb{R}^d) \rightarrow \mathbb{R}^d$ be a $\mathcal{F}_t$ measurable map. If the following limit exists
\begin{eqnarray*}
D_{\mathbb{V}} F_t(B) \  = \  \lim_{\e_1 \rightarrow 0} \frac{  F_t(B+\e_1 \mathbb{V}) - F_t(B)  }{\e_1 }
\end{eqnarray*}
in $L^2 (  (\Omega, \mathcal{F}, \mathbb{P}), \mathbb{R}^d  ) $, then $F_t(B)$ is said to be Malliavin differentiable and $D_{\mathbb{V}} F_t (B)$ is called the Malliavin derivative of $F_t(B)$ in the direction $\mathbb{V}$.

\textbf{Bismut's formula.}  For Malliavin differentiable $F_t(B)$ such that $F_t(B), D_{\mathbb{V}} F_t(B) \in L^2((\Omega,\mathcal{F},\mathbb{P}), \mathbb{R}^d)$, we have
\begin{eqnarray}\label{e:BisFor}
\mathbb{E} [D_{\mathbb{V}} F_t(B)]  \ = \   \mathbb{E}  \left[   F_t(B)  \int_0^t \langle  v(s), \dif B_s  \rangle \right].
\end{eqnarray}

The following Malliavin derivative of $X^{\e,x}_t$ along the direction $\mathbb{V}$ exists in $L^2((\Omega,\mathcal{F},\mathbb{P}), \mathbb{R}^d)$ and is defined by
\begin{eqnarray}\label{MaC1}
D_{\mathbb{V}} X^{\e,x}_t &=& \lim_{\e_1 \rightarrow 0 }\frac{  X^{\e,x}_t(B+\e_1 \mathbb{V}) - X^{\e,x}_t(B)  }{\e_1}.
\end{eqnarray}
It satisfies the following equation
\begin{eqnarray*}
D_{\mathbb{V}} X^{\e,x}_t &=& \sigma \mathbb{V}(t)  + \int_0^t \nabla g_{\e} (X_{s}^{\e,x}) D_{\mathbb{V}} X_{s}^{\e,x} \dif s, \quad D_{\mathbb{V}} X_{0}^{\e,x}=0,
\end{eqnarray*}
which is solved by
\begin{eqnarray*}
D_{\mathbb{V}} X^{\e,x}_t &=&  \int_0^t J_{r,t}^{\e,x} \sigma v(r) \dif r.
\end{eqnarray*}

Taking $v(r)=\frac{\sigma^{-1}}{ t}  J^{\e,x}_{r}u$ for $0\leq r \leq t$, by \eqref{hJaF1}, we get
\begin{eqnarray}  \label{e:DVNu}
D_{\mathbb{V}} X^{\e,x}_t&=& \nabla_{u} X^{\e,x}_{t} .
\end{eqnarray}
With the same $v$, a similar straightforward calculation gives that
\begin{eqnarray*}
D_{\mathbb{V}} X_{s}^{\e,x} & = &\frac{s}{t} \nabla_u X_{s}^{\e,x}, \quad 0 \le s \le t.
\end{eqnarray*}

For further use, for $x, u\in \R^d$, we define
\begin{eqnarray} \label{e:Iuxt}
\mathcal{I}_{u}^{\e,x}(t) \ : = \ \frac{1}{ t } \int_0^t \langle \sigma^{-1} J^{\e,x}_{r} u, \dif B_r \rangle, \ \ \ \ \ 
\mathcal{I}_{u}^{x}(t) \ : = \ \frac{1}{ t } \int_0^t \langle \sigma^{-1} J^{x}_{r} u, \dif B_r \rangle.
\end{eqnarray}

Now we are at the position to state the following lemmas.
\begin{lemma}\label{lem:EIm}
For all $x, u \in \R^d$, $m \ge 2$ and $t>0$, we have
\begin{eqnarray}
\E \left|\mathcal{I}_{u}^{\e,x}(t)\right|^m, \E \left|\mathcal{I}_{u}^{x}(t)\right|^m \ &\le& \ \frac{C |u|^{m}}{t^{m/2}} e^{mC_{ {\rm op} }t}, \label{e:IuexEst}  \\
	\lim_{\e \rightarrow 0} \E \left|\mathcal{I}_{u}^{\e,x}(t)- \mathcal{I}_{u}^{x}(t)\right|^m&=&0. \label{e:IueCon}
\end{eqnarray}
\end{lemma}

\subsubsection{Moment estimate for weighted occupation time} \label{app:occupation}

 We introduce in this section a weighted occupation time for the limiting diffusion $X_t^x$, and study its moment estimate. 
 This result is used in the proof of Lemma \ref{l:XeCon}, and we expect that the method in this section may be used in future research.

For any $\eps>0$, the weighted occupation time $L_t^{\eps,x}$ is defined as
\begin{eqnarray*}
L_t^{\eps,x} &=& \int_0^t \left[-\frac{1}{\eps^2} ({\rm e}^{\prime} X_s^x)^2 +1 \right] 1_{ \{ |{\rm e}^{\prime} X_s^x| \leq \eps \} } \dif s.
\end{eqnarray*}
We know $L_t^{\eps,x} \geq 0$ for all $t\geq 0$. We call $L_t^{\eps,x}$ weighted occupation time because it can be represented as an integral over an occupation measure associated with the limiting diffusion $X_t^x$. 
Namely, we have
\begin{eqnarray*}
L^{\eps,x}_t & = & \int_0^t \psi({\rm e}' X^x_s) 1_{\{|{\rm e}' X^x_s| \le \eps \}} \dif s  \ = \  \int_{|y| \le \eps} \psi(y) A^{x}_t(\dif y),
\end{eqnarray*}
where $\psi(y)=1-\frac{y^2}{\eps^2} $ and $A^{x}_t(\cdot):=\int_0^t  \delta_{  {\rm e}' X^{x}_s}   (\cdot) \dif s$ with $\delta_z(\cdot)$ being the delta function of a given $z$. $A^x_t(\cdot)$ is called the occupation measure of ${\rm e}' X^x_s$ over time $[0,t]$. 
If $\psi(x)=1$, then $L^{\eps,x}_t$ will be the occupation time of $({\rm e}' X^x_s)_{0 \le s \le t}$ on the set $\{|y| \le \eps \}$.

\begin{proposition}\label{lem:occupation}
For $L_t^{\eps,x}$ defined above, there exist some positive constants $C$ and $c$, independent of $\eps$ and $t$, such that
\begin{eqnarray*}
\E L_t^{\eps,x} &\leq&  C\eps e^{c t} (1+|x|)(1+t).
\end{eqnarray*}
\end{proposition}

\subsection{Proofs of Theorems \ref{thm:CLT} and \ref{thm:MDP}}
 We first prove Theorem \ref{thm:CLT} by It\^{o}'s formula and martingale CLT, and then Theorem \ref{thm:MDP} by a criterion from Wu \cite{WLM1}.
\begin{proof}[Proof of Theorem \ref{thm:CLT}]
For $(X_t)_{t\geq 0} $ in SDE \eqref{hSDEg} with $X_0=x$, using It\^{o}'s formula to $f$ which is the solution to Stein's equation \eqref{e:SE}, we have
\begin{eqnarray*}
f(X_t^x)-f(x)
&=&\int_0^t \mathcal{A}f(X_s^x)\dif s+\int_0^t  (\nabla f(X_s^x))^{\prime} \sigma  \dif B_s  \nonumber \\
&=&\int_0^t [ h(X_s^x)-\mu(h)] \dif s+\int_0^t  (\nabla f(X_s^x))^{\prime} \sigma \dif B_s,
\end{eqnarray*}
which implies that
	\begin{eqnarray*}
		\sqrt{t} \left[\frac{1}{t}\int_0^t \delta_{X_s^{x}}(h)\dif s-\mu(h) \right]
		&=&\frac{1}{\sqrt{t}}\int_0^t [ h(X_s^{x}) -\mu(h) ] \dif s  \\
		&=&\frac{1}{\sqrt{t}} [ f(X_t^x)-f(x)]-\frac{1}{\sqrt{t}} \int_0^t (\nabla f(X_s^x))^{\prime} \sigma \dif B_s.
	\end{eqnarray*}
From Lemma \ref{lem:regf} and the estimate for $\E V(X_t^x)$ in \eqref{e:Vm}, we obtain
	\begin{eqnarray*}
	\E \left| \frac{1}{\sqrt{t}} [ f(X_t^{x})-f(x)] \right|
	&\to& 0 { \ \ \rm as \ \ } t \to \infty.
	\end{eqnarray*}
It follows from Lemmas \ref{lem:regf} and \ref{lem:AV2} and inequality \eqref{e:BV} that  for some constant $C>0$ such that
\begin{eqnarray*}
\mu( |\sigma^{\prime} \nabla f|^2  ) \ \leq \ C\mu(V^4)  \ < \ \infty .
\end{eqnarray*}

We write
	\begin{eqnarray*}
\frac{1}{\sqrt{t}} \int_0^t (\nabla f(X_s^x))^{\prime} \sigma \dif B_s
\ = \
\frac{1}{\sqrt{t}} \left(\int_0^1 + \int_1^2 +\cdots + \int_{\lfloor t \rfloor-1}^{\lfloor t \rfloor} + \int_{\lfloor t \rfloor}^t \right) (\nabla f(X_s^x))^{\prime} \sigma \dif B_s,
	\end{eqnarray*}
and denote
	\begin{eqnarray*}
U_i \ = \  \int_{i-1}^i (\nabla f(X_s^x))^{\prime} \sigma \dif B_s \text{ \ \ for \ \ } i=1,2,\cdots,\lfloor t \rfloor
	\end{eqnarray*}
and
	\begin{eqnarray*}
U_{\lfloor t \rfloor+1} \ = \  \int_{\lfloor t \rfloor}^t (\nabla f(X_s^x))^{\prime} \sigma \dif B_s.
	\end{eqnarray*}
Then one obtains that $U_i$'s are martingale differences and $\E U_i^2 <\infty$ for all $i=1,2,\cdots,\lfloor t \rfloor+1$. We claim that
\begin{eqnarray}
\lim_{t \to \infty} \E\left( \max_{1\leq i \leq \lfloor t \rfloor +1 } \frac{1}{\mu( |\sigma^{\prime} \nabla f|^2 ) t } |U_i|^2 \right) \ = \ 0, \label{e:CLT1} \\
\lim_{t \to \infty} \E \left| \sum_{i=1}^{\lfloor t \rfloor +1 } \frac{1}{\mu( |\sigma^{\prime} \nabla f|^2 ) t } |U_i|^2 -1 \right|^2 \ = \ 0. \label{e:CLT2}
\end{eqnarray}
We observe that  equalities \eqref{e:CLT1} and \eqref{e:CLT2} imply
\begin{eqnarray*}
\E\left( \max_{1\leq i \leq \lfloor t \rfloor +1 } \frac{1}{ \sqrt{\mu( |\sigma^{\prime} \nabla f|^2 ) t} } |U_i| \right) \to 0
\text{ \ \ and \ \ }
\sum_{i=1}^{\lfloor t \rfloor +1 } \frac{1}{\mu( |\sigma^{\prime} \nabla f|^2 ) t } |U_i|^2 \stackrel{p}{\longrightarrow} 1
	\end{eqnarray*}
as $t$ goes to infinity. Using the martingale CLT in Sethuraman \cite[Theorem 2]{SS1} due to McLeish \cite{MDL1}, one has
\begin{eqnarray*}
\frac{1}{\sqrt{\mu(|\sigma^{\prime}  \nabla f|^2) t}}\sum_{i=1}^{\lfloor t \rfloor+1} U_i
\ = \ \frac{1}{\sqrt{\mu(|\sigma^{\prime}  \nabla f|^2) t}} \int_0^t  (\nabla f(X_s^x))^{\prime} \sigma \dif B_s  \ \Rightarrow  \ \mathcal{N} (0,1) \text{ \ \ as \ \ } t \to \infty.
	\end{eqnarray*}
Then, we obtain
	\begin{eqnarray*}
		\sqrt{t} \left[ \frac{1}{t}\int_0^t \delta_{X_s^{x}}(h)\dif s-\mu(h) \right]
		\ \Rightarrow \ \mathcal{N} (0,\mu(|\sigma^{\prime}  \nabla f|^2))
\text{ \ \ as \ \ } t \to \infty.
	\end{eqnarray*}

It remains to show that equalities \eqref{e:CLT1} and \eqref{e:CLT2} hold. For \eqref{e:CLT1}, one has
\begin{eqnarray*}
\E\left( \max_{1\leq i \leq \lfloor t \rfloor +1} |U_i|^2 \right)
&=& \E\left( \max_{1\leq i \leq \lfloor t \rfloor +1} (|U_i|^2 1_{ \{ |U_i|^2\leq \sqrt{t} \} } + |U_i|^2 1_{ \{|U_i|^2>\sqrt{t}  \}})  \right)  \nonumber \\
&\leq& \E\left( \max_{1\leq i \leq \lfloor t \rfloor+1} |U_i|^2 1_{ \{ |U_i|^2\leq \sqrt{t} \} } \right)
+ \E\left( \max_{1\leq i \leq \lfloor t \rfloor+1} |U_i|^2 1_{ \{|U_i|^2> \sqrt{t}  \}}  \right) \nonumber \\
&\leq& \sqrt{t} + \E \left( \max_{1\leq i \leq \lfloor t \rfloor +1} |U_i|^2 1_{ \{|U_i|^2> \sqrt{t} \}}  \right) \nonumber \\
&\leq& \sqrt{t} +(\lfloor t \rfloor+1) \max_{1\leq i \leq \lfloor t \rfloor+1} \E (|U_i|^2 1_{ \{|U_i|^2> \sqrt{t} \}} ),
\end{eqnarray*}
where the last inequality holds from that
\begin{eqnarray*}
\E \left( \max_{1\leq i \leq \lfloor t \rfloor +1} |U_i|^2 1_{ \{|U_i|^2> \sqrt{t} \}}  \right)  \ &\leq& \  \sum_{i=1}^{ \lfloor t \rfloor +1} \E (|U_i|^2 1_{ \{|U_i|^2> \sqrt{t} \}} ).
\end{eqnarray*}
Thus, we obtain
\begin{eqnarray*}
&& \E \left( \max_{1\leq i \leq \lfloor t \rfloor+1 } \frac{1}{ \mu( |\sigma^{\prime} \nabla f|^2 ) t} | U_i |^2 \right)  \\
&\leq& \frac{1}{ \mu( |\sigma^{\prime} \nabla f|^2 ) t} \left[ \sqrt{t} +(\lfloor t \rfloor+1) \max_{1\leq i \leq \lfloor t \rfloor +1} \E (|U_i|^2 1_{ \{|U_i|^2> \sqrt{t} \}} ) \right] \\
&\to & 0 \text{ \ \ as \ \ } t\to \infty.
\end{eqnarray*}

For \eqref{e:CLT2}, we have
\begin{eqnarray*}
&& \E \left| \sum_{i=1}^{ \lfloor t \rfloor +1 } \frac{1}{\mu( |\sigma^{\prime} \nabla f|^2 ) t } |U_i|^2 -\frac{ \lfloor t \rfloor +1 }{t} \right|^2
\  =  \ \E \left| \frac{1}{t} \sum_{i=1}^{ \lfloor t \rfloor +1 } \left( \frac{1}{\mu( |\sigma^{\prime} \nabla f|^2 )  } |U_i|^2 - 1 \right) \right|^2 \\
&=& \frac{1}{t^2} \sum_{i=1}^{ \lfloor t \rfloor +1 }\E\left( \frac{1}{\mu( |\sigma^{\prime} \nabla f|^2 )  } |U_i|^2 - 1 \right)^2 \\
&&+\frac{2}{t^2}\sum_{i<j}\E\left[ ( \frac{1}{\mu( |\sigma^{\prime} \nabla f|^2 )  } |U_i|^2 - 1 ) ( \frac{1}{\mu( |\sigma^{\prime} \nabla f|^2 )  } |U_j|^2 - 1 )  \right]  \\
&=:& {\rm I} + {\rm II}.
\end{eqnarray*}
Using the Burkholder-Davis-Gundy inequality, there exists some positive constant $C$ such that
\begin{eqnarray*}
\E|U_i|^4
&=& \E \left| \int_{i-1}^i (\nabla f(X_s^x))^{\prime} \sigma \dif B_s \right|^4
\ \leq \ C \E \left| \int_{i-1}^i |(\nabla f(X_s^x))^{\prime} \sigma|^2 \dif s \right|^2.
\end{eqnarray*}
Combining with Lemma \ref{lem:regf}, we know
\begin{eqnarray*}
\E\left( \frac{1}{\mu( |\sigma^{\prime} \nabla f|^2 )  } |U_i|^2 - 1 \right)^2
&=& \E \left[ \frac{1}{\mu^2(|\sigma^{\prime} \nabla f|^2)} |U_i|^4 - \frac{2}{\mu( |\sigma^{\prime}\nabla f|^2)}|U_i|^2+1 \right]  \\
&\leq & C(1+\E |X_i^x|^8).
\end{eqnarray*}
Combining this with \eqref{e:Vm} in Lemma \ref{lem:AV2} and inequality  \eqref{e:BV}, we know
\begin{eqnarray*}
{\rm I}
\ &=& \ \frac{1}{t^2} \sum_{i=1}^{ \lfloor t \rfloor +1 } \E\left( \frac{1}{\mu( |\sigma^{\prime} \nabla f|^2 )  } |U_i|^2 - 1 \right)^2
\ \leq \  \frac{C}{t^2} \sum_{i=1}^{ \lfloor t \rfloor +1 }  (1+\E |X_i^x|^8)
\to 0, {\ \ \rm as \ \ } t \to \infty.
\end{eqnarray*}
We also have
\begin{eqnarray*}
&& {\rm II}
\ = \  \frac{2}{t^2}\sum_{i<j}\E \left[ \left( \frac{1}{\mu( |\sigma^{\prime} \nabla f|^2 ) } |U_i|^2 - 1 \right) \left( \frac{1}{\mu( |\sigma^{\prime} \nabla f|^2 )  } |U_j|^2 - 1 \right) \right] \\
&=& \frac{2}{t^2}\sum_{i=1}^{\lfloor t \rfloor } \sum_{j=i+1}^{\lfloor t \rfloor +1} \E \left\{ \left(\frac{1}{\mu(|\sigma^{\prime} \nabla f|^2)}|U_i|^2-1\right) \E \left[ \left ( \frac{1}{\mu( |\sigma^{\prime} \nabla f|^2 )  } |U_j|^2 - 1 \right) |\mathcal{F}_i  \right]  \right\} \\
&=& \frac{2}{\mu( |\sigma^{\prime} \nabla f|^2 ) t^2}\sum_{i=1}^{\lfloor t \rfloor } \sum_{j=i+1}^{\lfloor t \rfloor +1} \E \left\{ \left(\frac{1}{\mu(|\sigma^{\prime} \nabla f|^2)}|U_i|^2-1\right) \E[(|U_j|^2 - \mu( |\sigma^{\prime} \nabla f|^2 )  ) |\mathcal{F}_i  ] \right\} \\
&=& \frac{2}{\mu( |\sigma^{\prime} \nabla f|^2 ) t^2}\sum_{i=1}^{\lfloor t \rfloor } \sum_{j=i+1}^{\lfloor t \rfloor +1} \E\left\{ \left(\frac{1}{\mu(|\sigma^{\prime} \nabla f|^2)}|U_i|^2-1\right) \int_{j-1}^j [ \E|\sigma^{\prime} \nabla f(X_s^{X_i^x})|^2 - \mu( |\sigma^{\prime} \nabla f|^2 ) ] \dif s \right\} \\
&=& \frac{2}{\mu( |\sigma^{\prime} \nabla f|^2 ) t^2}\sum_{i=1}^{\lfloor t \rfloor } \E \left\{ \left(\frac{1}{\mu(|\sigma^{\prime} \nabla f|^2)}|U_i|^2-1\right) \int_{i}^{\lfloor t \rfloor+1} [ \E|\sigma^{\prime} \nabla f(X_s^{X_i^x})|^2 - \mu( |\sigma^{\prime} \nabla f|^2 ) ] \dif s \right\}.
\end{eqnarray*}
It follows from Lemma \ref{lem:AV2} that
\begin{eqnarray*}
{\rm II} &\leq& \frac{C}{t^2}\sum_{i=1}^{\lfloor t \rfloor } \E \left\{ \left|\frac{1}{\mu(|\sigma^{\prime} \nabla f|^2)}|U_i|^2-1 \right| \int_{0}^{\lfloor t \rfloor+1} (1+V^2(X_i^x))e^{-cs} \dif s \right\} \\
&\leq& \frac{C}{t^2}\sum_{i=1}^{\lfloor t \rfloor } \E \left\{ \left|\frac{1}{\mu(|\sigma^{\prime} \nabla f|^2)}|U_i|^2-1 \right| (1+V^2(X_i^x))  \right\} \\
&\leq& \frac{C}{t^2}\sum_{i=1}^{\lfloor t \rfloor } \left[\E \left|\frac{1}{\mu(|\sigma^{\prime} \nabla f|^2)}|U_i|^2-1\right|^2\right]^{\frac{1}{2}} [1+\E V^4(X_i^x)]^{\frac{1}{2}} \\
&\to& 0, {\ \ \rm as \ \ } t \to \infty.
\end{eqnarray*}
Combining estimates for ${\rm I}$ and ${\rm II}$, we obtain
\begin{eqnarray*}
\lim_{t \to \infty} \E \left| \sum_{i=1}^{ \lfloor t \rfloor +1 } \frac{1}{\mu( |\sigma^{\prime} \nabla f|^2 ) t } |U_i|^2 -\frac{ \lfloor t \rfloor +1 }{t} \right|^2  \  =  \ 0.
\end{eqnarray*}
Thus, \eqref{e:CLT2} holds.  The proof is complete.
\end{proof}

\medskip

\begin{proof}[Proof of Theorem \ref{thm:MDP}]
	It follows from Lemma \ref{lem:AV2} that for any function $\tl{f}$ satisfying $\tl{f}(x)\leq 1+V(x)$ for all $x\in \R^d$ with $V$ in \eqref{e:Lypfun}, there exist some positive constants $C$ and $c$ such that
\begin{eqnarray*}
| P_t \tl{f}(x) - \mu(\tl{f}) |  \ \leq \ C(1+V(x)) e^{-c t}.
	\end{eqnarray*}
	It follows from Wu \cite[Remark (2.17)]{WLM2} that $P_t$ has a spectral gap near its largest eigenvalue $1$ which implies that $1$ is an isolate eigenvalue. Since $P_t c = c$ for all $t>0$, one has the property that $1$ is an eigenvalue of $P_t$ for all $t>0$. If there exists some function $\hat{f}$ satisfying $0\neq \hat{f}(x) \leq 1+V(x)$ for all $x\in \R^d$ such that  $P_{t_1} \hat{f} = \lambda \hat{f}$ for some $t_1 > 0$ and $\lambda \in \mathbb{C}$ with $|\lambda|=1$, then $\lambda^{\tl{n}} \hat{f} = P_{\tl{n} t_1} \hat{f} \to \mu(\hat{f})$ as $\tl{n} \to \infty$, so that $\hat{f}$ has to be constant and $\lambda=1$. Thus, $1$ is a simple and the only eigenvalue with modulus $1$ for $P_t$.

For $h\in \mathcal{B}_b(\R^d, \R)$, it follows from Wu \cite[Theorem 2.1]{WLM1} that
$$\PP\left( \frac{1}{a_t\sqrt{t}}\int_0^t [ h(X_s^x) - \mu(h)  ]  \dif s \in  \,\, \cdot \,\, \right)$$
satisfies the large deviation principle with speed $a_t^{-2}$ and rate function $I_h(z)=\frac{z^2}{2\mathcal{V}(h)}$ with
$$\mathcal{V}(h) \ = \ 2\int_0^{\infty} \langle P_t h, h-\mu(h) \rangle_{\mu} \dif t,$$
that is,
\begin{eqnarray*}
	-\inf_{z \in A^{ {\rm o} } } I_h (z)
&\leq& \liminf_{t\to\infty}\frac{1}{a_t^2}\log \mathbb{P} \left( \frac{\sqrt{t}}{a_t}  \left[\mcl E^x_t(h)-\mu(h)\right] \in A \right)  \\
&\leq& \limsup_{t\to\infty}\frac{1}{a_t^2}\log \mathbb{P} \left( \frac{\sqrt{t}}{a_t}  \left[\mcl E^x_t(h)-\mu(h)\right] \in A \right)
\ \leq  \ -\inf_{z \in \bar{A}} I_h (z),
	\end{eqnarray*}
	where $\bar{A}$ and $A^{ {\rm o} }$ are the closure and  interior of set $A$ respectively.

We claim that
\begin{eqnarray}\label{e:Vh}
\mathcal{V}(h) \ = \ \mu(|\sigma^{\prime}\nabla f|^2 ),
	\end{eqnarray}
where $f$ is the solution to Stein's equation \eqref{e:SE}. Then the desired result holds.

Now, we show that the claim \eqref{e:Vh} holds. With some calculations, one has
\begin{eqnarray*}
&& \frac{1}{t}\E^{\mu} \left( \int_0^t [h(X_s^x) - \mu(h)] \dif s \right)^2  \\
&=& \frac{1}{t}\E^{\mu} \left( \int_0^t \int_0^t [h(X_s^x)-\mu(h)][h(X_u^x)-\mu(h)] \dif u \dif s  \right) \\
&=& \frac{2}{t}\E^{\mu} \left( \int_0^t \int_0^u [h(X_s^x)-\mu(h)][h(X_u^x)-\mu(h)] \dif s \dif u \right) \\
&=& \frac{2}{t} \int_0^t \int_0^u \E^{\mu} \{ [h(X_s^x)-\mu(h)] \E \{ [h(X_u^x)-\mu(h)]| X_s^x \} \} \dif s \dif u \\
&=& \frac{2}{t} \int_0^t \int_0^u \E^{\mu} \{ h(X_s^x) \E \{ [h(X_u^x)-\mu(h)]| X_s^x \} \} \dif s \dif u,
	\end{eqnarray*}
where the third equality holds from conditional probability and the last equality holds from
	\begin{eqnarray*}
\int_0^t \int_0^u \E^{\mu} \{ \mu(h) \E \{ [h(X_u^x)-\mu(h)]| X_s^x \} \} \dif s \dif u
&=& \mu(h) \int_0^t \int_0^u \E^{\mu} \{  h(X_u^x)-\mu(h) \} \dif s \dif u \\
&=&0.
	\end{eqnarray*}

Furthermore, for all $0\leq s \leq u<\infty$, one has
\begin{eqnarray*}
\E^{\mu} \{ h(X_s^x) \E \{ [h(X_u^x)-\mu(h)]| X_s^x \} \}
&=&  \E^{\mu} \{ h(X_s^x) [ P_{u-s}h(X_s^x) - \mu(h) ]   \} \\
& = & \int_{\R^d} h(y)[ P_{u-s}h(y)-\mu(h) ] \mu(\dif y),
	\end{eqnarray*}
which implies that
	\begin{eqnarray*}
&& \frac{2}{t} \int_0^t \int_0^u \E^{\mu} \{ h(X_s^x) \E \{ [h(X_u^x)-\mu(h)]| X_s^x \} \} \dif s \dif u \\
&=& \frac{2}{t} \int_0^t\int_0^u \int_{\R^d} h(y)[ P_{u-s}h(y)-\mu(h) ] \mu(\dif y) \dif s \dif u \\
&=& \int_{\R^d} h(y) \frac{2}{t} \int_0^t\int_0^u [ P_{u-s}h(y)-\mu(h) ]  \dif s \dif u \mu(\dif y) \\
&=& \int_{\R^d} h(y) \frac{2}{t} \int_0^t\int_0^u [ P_{\tl{s}}h(y)-\mu(h) ]  \dif \tl{s} \dif u \mu(\dif y).
\end{eqnarray*}
Using the Hospital's rule, one has
\begin{eqnarray*}
&& \lim_{t\to\infty}\int_{\R^d} h(y) \frac{2}{t} \int_0^t\int_0^u [ P_{\tl{s}}h(y)-\mu(h) ]  \dif \tl{s} \dif u \mu(\dif y) \\
&=&\lim_{t\to\infty}2 \int_{\R^d} h(y) \int_0^{t} [P_{\tl{s}} h(y) - \mu(h)] \dif \tl{s} \mu(\dif y) \\
&=& 2 \int_{\R^d} h(y) \int_0^{\infty} [P_{\tl{s}} h(y) - \mu(h)] \dif \tl{s} \mu(\dif y).
\end{eqnarray*}

From the expression of $\mathcal{V}(h)$, we know
\begin{eqnarray*}
\mathcal{V}(h)
&=& 2\int_0^{\infty} \langle P_t h, h-\mu(h) \rangle_{\mu} \dif t
\ = \ 2\int_0^{\infty} \langle P_t h - \mu(h), h \rangle_{\mu} \dif t,
\end{eqnarray*}
so that,
\begin{eqnarray*}
\frac{1}{t}\E^{\mu} \left( \int_0^t [h(X_s^x) - \mu(h)] \dif s \right)^2
&\to& \mathcal{V}(h)
\text{ \ \ as \ \ } t \to \infty.
	\end{eqnarray*}

Using It\^{o}'s formula and Stein's equation \eqref{e:SE}, we have
\begin{eqnarray*}
f(X_t^x)-f(x)
&=&\int_0^t \mathcal{A}f(X_s^x)\dif s+\int_0^t  (\nabla f(X_s^x))^{\prime} \sigma  \dif B_s  \nonumber \\
&=&\int_0^t [ h(X_s^x)-\mu(h)] \dif s+\int_0^t  (\nabla f(X_s^x))^{\prime} \sigma \dif B_s,
\end{eqnarray*}
which implies that
\begin{eqnarray*}
\frac{1}{t}\E^{\mu} \left( \int_0^t [h(X_s^x) - \mu(h)] \dif s \right)^2
&=& \E^{\mu} \left( \frac{1}{\sqrt{t}} [ f(X_t^x)-f(x)] - \frac{1}{\sqrt{t}} \int_0^t  (\nabla f(X_s^x))^{\prime} \sigma \dif B_s \right)^2 \\
&\to& \mu(|\sigma^{\prime} \nabla f |^2) \text{ \ \ as \ \ } t \to \infty.
	\end{eqnarray*}
Thus, we know $\mathcal{V}(h) = \mu( |\sigma^{\prime} \nabla f|^2 )$. The claim \eqref{e:Vh} holds. The proof is complete.
\end{proof}

\section{ Proof of Theorem \ref{thm:EMMDP} } \label{sec:EMCLTMDP}
In order to determine the variance in Theorem \ref{thm:EMMDP}, we need to study the third Stein's equation as the following 
\begin{equation} \label{e:Stein-3}
\mathcal A_\eta f(x)=h(x)-\tilde{\mu}_\eta(h),
\end{equation} 
where $\mathcal A_\eta f(x)=\tl{\mathcal{P}}_{\eta} f(x)-f(x)$ with $\tl{\mathcal{P}}_{\eta} f(x)=\E f(\tl X^{\eta,x}_1)$ for all $x \in \R^d$. It is easy to verify that the solution of the equation \eqref{e:Stein-3} is 
\begin{eqnarray} \label{e:fRep}
f(x) \ &=&-\sum_{k=0}^{\infty} \tl{\mathcal{P}}^k_{\eta} \left( h(x)-\tl{\mu}_{\eta}(h) \right),
\end{eqnarray}
because of $ \tl{\mathcal{P}}_{\eta} f(x) = -\sum_{k=1}^{\infty} \tl{\mathcal{P}}^k_{\eta} \left( h(x)-\tl{\mu}_{\eta}(h) \right)$.

\begin{proof}[Proof of Theorem \ref{thm:EMMDP}]
(i) From Proposition \ref{p:GeneralErgodicEM}, we know the Markov chain $(\tl{X}_{k}^{\eta})_{k\in \mathbb{N}_0}$ is exponentially ergodic under total variation distance and $\tl{\mu}_{\eta}(V^{\ell})\leq C$ for any integers $\ell$ and $V$ in \eqref{e:Lypfun}. It follows from Jones \cite[Theorem 9]{jones2004markov} that for any function $h$ satisfying $|h|\leq V^{\ell}$ with some integer $\ell$ and any initial distribution $\tl{X}_0^{\eta}$, one has
     \begin{eqnarray*}
\sqrt{n}\left[ \mcl E_n^{\eta,x}(h)  -  \tl{\mu}_{\eta}(h)  \right] \  \Rightarrow \ \mathcal{N}(0, \sigma_h^2)
	\end{eqnarray*}
with
\begin{eqnarray}\label{e:sigmaf2}
\sigma_h^2  
&=& \langle h-\tl{\mu}_{\eta}(h), h-\tl{\mu}_{\eta}(h)   \rangle_{\tl{\mu}_{\eta}} + 2\sum_{k=1}^{\infty} \langle \tl{\mathcal{P}}^k_{\eta} h, h-\tl{\mu}_{\eta}(h) \rangle_{\tl{\mu}_{\eta}}.
\end{eqnarray}

We give some calculations for the variance of $\sqrt{n}\left[ \mcl E_n^{\eta,x}(h)  -  \tl{\mu}_{\eta}(h)  \right]$. Since $\E [\sqrt{n} ( \mcl E_n^{\eta,x}(h)  -  \tl{\mu}_{\eta}(h) ) ]=0$, one has 
\begin{eqnarray}\label{e:var-d1}
 \E |\sqrt{n}\left[ \mcl E_n^{\eta,x}(h)  -  \tl{\mu}_{\eta}(h)  \right]|^2
&=& \E| \frac{1}{\sqrt{n}} \sum_{k=0}^{n-1} [h(\tl{X}^{\eta,x}_k) -  \tl{\mu}_{\eta}(h) ]  |^2  \nonumber \\
&=& \frac{1}{n} \sum_{k=0}^{n-1}\E | h(\tl{X}^{\eta,x}_k) -  \tl{\mu}_{\eta}(h)  |^2 \nonumber \\
&& + \frac{2}{n} \sum_{i=0}^{n-1} \sum_{j=0}^{i-1} 
\E \{   [h(\tl{X}^{\eta,x}_i) -  \tl{\mu}_{\eta}(h) ]  [h(\tl{X}^{\eta,x}_j) -  \tl{\mu}_{\eta}(h) ]      \}.
\end{eqnarray}	
We aim to the second term on the right hand of above equality \eqref{e:var-d1}. By the conditional expectation, one has 
\begin{eqnarray}\label{e:var-d2}
&& \frac{2}{n} \sum_{i=0}^{n-1} \sum_{j=0}^{i-1}  \E \{   [h(\tl{X}^{\eta,x}_i) -  \tl{\mu}_{\eta}(h) ]  [h(\tl{X}^{\eta,x}_j) -  \tl{\mu}_{\eta}(h) ]   \}  \nonumber \\
&=& \frac{2}{n} \sum_{i=0}^{n-1} \sum_{j=0}^{i-1}  \E \{   [h(\tl{X}^{\eta,x}_j) -  \tl{\mu}_{\eta}(h) ]  \E[ (h(\tl{X}^{\eta,x}_i) -  \tl{\mu}_{\eta}(h) ) | \tl{X}^{\eta,x}_j  ]    \}  \nonumber \\
&=&  \frac{2}{n} \sum_{j=0}^{n-2}  \E \{   [h(\tl{X}^{\eta,x}_j) -  \tl{\mu}_{\eta}(h) ]  \E[ \sum_{i=j+1}^{n-1}  (h(\tl{X}^{\eta,x}_i) -  \tl{\mu}_{\eta}(h) ) | \tl{X}^{\eta,x}_j  ]    \}   \nonumber  \\
&=&  \frac{2}{n} \sum_{j=0}^{n-2}  \E \{   [h(\tl{X}^{\eta,x}_j) -  \tl{\mu}_{\eta}(h) ]  \E[ \sum_{k=1}^{n-1-j} \tl{P}_{\eta}^k (h(\tl{X}^{\eta,x}_j) -  \tl{\mu}_{\eta}(h) ) ]    \}  \nonumber  \\
&=&  \frac{2}{n} \sum_{j=0}^{n-2}  \E \{   [h(\tl{X}^{\eta,x}_j) -  \tl{\mu}_{\eta}(h) ] [ \sum_{k=1}^{n-1-j} \tl{P}_{\eta}^k (h(\tl{X}^{\eta,x}_j) -  \tl{\mu}_{\eta}(h) ) ]    \}  \nonumber  \\
&\to& 2\sum_{k=1}^{\infty} \langle \tl{\mathcal{P}}^k_{\eta} [h(x)-\tl{\mu}_{\eta}(h)], h(x)-\tl{\mu}_{\eta}(h) \rangle_{\tl{\mu}_{\eta}}, \quad {\rm as \ \ } n\to \infty.
\end{eqnarray}

Combining \eqref{e:var-d1} and \eqref{e:var-d2},  we know 
\begin{eqnarray*}
\E |\sqrt{n}\left[ \mcl E_n^{\eta,x}(h)  -  \tl{\mu}_{\eta}(h)  \right]|^2
&\to&  \langle h-\tl{\mu}_{\eta}(h), h-\tl{\mu}_{\eta}(h)   \rangle_{\tl{\mu}_{\eta}} + 2\sum_{k=1}^{\infty} \langle \tl{\mathcal{P}}^k_{\eta} h, h-\tl{\mu}_{\eta}(h) \rangle_{\tl{\mu}_{\eta}}, 
\end{eqnarray*}	
as $n\to \infty$, which implies the equality \eqref{e:sigmaf2} holds.

Since $\langle \tl{\mu}_{\eta}(h), h-\tl{\mu}_{\eta}(h) \rangle_{\tl{\mu}_{\eta}}=0$ and \eqref{e:fRep}, it is easy to verify that  
\begin{eqnarray*}
\tl{ \mcl V} (h) &=&  \langle f,  f \rangle_{\tl{\mu}_{\eta}}-\langle \tl{\mathcal{P}}_{\eta} f, \tl{\mathcal{P}}_{\eta} f \rangle_{\tl{\mu}_{\eta}} \\  
&=& \langle h-\tl{\mu}_{\eta}(h), h-\tl{\mu}_{\eta}(h)   \rangle_{\tl{\mu}_{\eta}} + 2\sum_{k=1}^{\infty} \langle \tl{\mathcal{P}}^k_{\eta} h, h-\tl{\mu}_{\eta}(h) \rangle_{\tl{\mu}_{\eta}}
\ \ = \ \ \sigma_h^2.
\end{eqnarray*}

(ii) It follows from the proof of Proposition \ref{p:GeneralErgodicEM} that there exist some positive constants $C$ and $c$ such that
\begin{eqnarray*}
\|\tilde{\mathcal{P}}_{\eta}^n(x,\cdot) -\tl{\mu}_{\eta} \|_{1+V} \ := \ \sup_{|f|\leq 1+V} | \tilde{\mathcal{P}}_{\eta}^n f(x)-\tl{\mu}_{\eta}(f)|
\ \leq \ C\eta^{-1} e^{-c n\eta}.
	\end{eqnarray*}
	It follows from Wu \cite[Remark (2.17)]{WLM2} that $\tilde{\mathcal{P}}^n_{\eta}$ has a spectral gap near its largest eigenvalue $1$ which implies that $1$ is an isolate eigenvalue. Since $\tilde{\mathcal{P}}^k_{\eta} c = c$ for all $k \in \mathbb{N}_0$, one has  that $1$ is an eigenvalue of $\tilde{\mathcal{P}}^k_{\eta}$ for all $k \in \mathbb{N}_0$. If there exists some function $\hat{f}$ satisfying $0\neq \hat{f}(x) \leq 1+V(x)$ for all $x\in \R^d$ such that  $\tilde{\mathcal{P}}^{k_1}_{\eta} \hat{f} = \lambda \hat{f}$ for some $k_1 \in \mathbb{N}_0$ and $\lambda \in \mathbb{C}$ with $|\lambda|=1$, then $\lambda^{\tl{n}} \hat{f} = \tilde{\mathcal{P}}^{\tl{n} k_1}_{\eta} \hat{f} \to \mu(\hat{f})$ as $\tl{n} \to \infty$, so that $\hat{f}$ has to be constant and $\lambda=1$. Thus, $1$ is a simple and the only eigenvalue with modulus $1$ for $\tilde{\mathcal{P}}_{\eta}^n$.

For $h\in \mathcal{B}_b(\R^d, \R)$, it follows from Wu \cite[Theorem 2.1]{WLM1} that
$$\PP\left( \frac{\sqrt{n}}{a_n}  \left[ \mcl E_n^{\eta,x}(h) - \tl{\mu}_{\eta}(h)\right]  \in \,\, \cdot \,\, \right)$$
satisfies the large deviation principle  with speed $a_n^{-2}$ and rate function $I_h(z)=\frac{z^2}{2\tl{ \mcl V} (h)}$ with $\tl{ \mcl V}(h)$ in \eqref{e:EMVh}, that is,
	\begin{eqnarray*}
	-\inf_{z \in A^{ {\rm o} } } \frac{ z^2 }{2 \tl{ \mcl V}(h)}
	\ &\leq& \
	\liminf_{n \to\infty}\frac{1}{a_n^2}\log \mathbb{P} \left( \frac{\sqrt{n}}{a_n}  \left[ \mcl E_n^{\eta,x}(h) - \tl{\mu}_{\eta}(h)\right]    \in A\right)   \\
	\ &\leq& \
	\limsup_{ n \to\infty}\frac{1}{a_n^2}\log \mathbb{P} \left( \frac{\sqrt{n}}{a_n}  \left[ \mcl E_n^{\eta,x}(h) - \tl{\mu}_{\eta}(h)\right]  \in A\right)
\ \leq \ -\inf_{z \in \bar{A}} \frac{ z^2 }{ 2 \tl{ \mcl V}(h) },
	\end{eqnarray*}
	where $\bar{A}$ and $A^{ {\rm o} }$ are the closure and  interior of set $A$ respectively. The proof is complete.
\end{proof}

\begin{appendix}
\section{ Exponential ergodicity for Markov chain $(\tl{X}_k^{\eta})_{k\in \mathbb{N}_0}$ }  \label{App:GeneralErgodicEM}

It follows from Dieker and Gao \cite[Theorem 1]{DG1} that there exists a positive definite matrix $\tilde{Q}=( \tilde{Q}_{ij}  )_{d\times d}$ with $\sum_{i,j=1}^d |\tl{Q}_{ij}|=1$ such that
\begin{eqnarray*}
\tilde{Q}(-R) + (-R)^{\prime} \tilde{Q} \ &<&  \ 0, \\
\tilde{Q}(-(I-p{\rm e}^{\prime})R) +(-R^{\prime}(I-{\rm e}p^{\prime}))\tilde{Q}
\ &\leq& \ 0.
\end{eqnarray*}
Let  the function $\tl{V}\in \mathcal{C}^2(\R^d,\R_+)$ be constructed in  Dieker and Gao \cite[Eq. (5.24)]{DG1}, that is,
\begin{eqnarray*}
\tl{V}(y)&=&({\rm e}^\prime y)^2+\kappa[y-p\phi({\rm e}^\prime y)]^\prime \tilde{Q}[y-p\phi({\rm e}^\prime y)], \qquad  \forall y\in \R^d,
\end{eqnarray*}
where $\kappa$ is a positive constant and $\phi\in \mathcal{C}^2(\R, \R)$ is a real-valued function which is defined as below:
\begin{eqnarray*}
\phi(z) &=&
\left\{
\begin{array}{lll}
z,                                   &  \text{if }  z \ge 0,  \\
-\frac{1}{2},                        & \text{if }   z \leq -1, \\
-\frac{1}{2}z^4 - z^3 + z,           &\text{if }     -1<z<0,
\end{array}
\right.
\end{eqnarray*}
it is easy to check that
\begin{eqnarray*}
(\nabla \tl{V}(y) )^{\prime} &=& 2({\rm e}^{\prime} y) {\rm e}^{\prime} + 2\kappa [ y^{\prime} - p^{\prime} \phi ({\rm e}^{\prime} y) ] \tilde{Q} [ I - p {\rm e}^{\prime} \dot{\phi}({\rm e}^{\prime} y)] \text{ \ \ for all \ \ } y \in \R^d.
\end{eqnarray*}
Then there exists some positive constant $C$ such that
\begin{eqnarray}\label{e:NtlV}
|\nabla \tl{V}(y)|
\ &\leq& \ C(1+|y|) \text{ \ \ for all \ \ } y \in \R^d.
\end{eqnarray}

The Lyapunov condition holds for function $\tl{V}$ from Dieker and Gao \cite[proof of Theorem 3]{DG1}, that is, there exist some positive constants $c_1$ and $\check{c}_1$ satisfying
\begin{eqnarray}\label{e:AtlV}
\mathcal{A}\tl{V}(y)
\ &\leq& \ -c_1 \tl{V}(y)+\check{c}_1
\text{ \ \ for any \ \ }y \in \R^d,
\end{eqnarray}
where $\mathcal{A}$ is in \eqref{e:A}. Also, it follows from Dieker and Gao \cite[proof of Theorem 3]{DG1} that there exist some positive constants $\hat{C}_1, \hat{C}_2$, $\hat{c}_1$ and $\hat{c}_2$ such that
\begin{eqnarray}\label{e:BtlV}
\hat{c}_1|y|^2 - \hat{c}_2
\  \leq \
\tl{V}(y)
\ \leq  \ \hat{C}_1|y|^2+\hat{C}_2
\text{ \ \ for all \ \ } y \in \R^d.
\end{eqnarray}

Let the Lyapunov function $V$ be defined as
\begin{eqnarray}\label{e:Lypfun}
V(y) &=& \tl{V}(y)+\hat{c}_2  \nonumber  \\
&=& ({\rm e}^\prime y)^2+\kappa[y-p\phi({\rm e}^\prime y)]^\prime \tilde{Q}[y-p\phi({\rm e}^\prime y)] + \hat{c}_2, \qquad  \forall y\in \R^d,
\end{eqnarray}
where $\hat{c}_2$ is in \eqref{e:BtlV}. Then one has
\begin{eqnarray}\label{e:BV}
\hat{c}_1|y|^2
\ \leq \
V(y)
\ \leq  \ \hat{C}_1|y|^2+\hat{C}_2+\hat{c}_2
\text{ \ \ for all \ \ } y \in \R^d.
\end{eqnarray}
Furthermore, we know
\begin{eqnarray*}
\nabla V(y) \  =  \ \nabla \tl{V}(y),
\qquad
\nabla^2 V(y) \ = \ \nabla^2 \tl{V}(y)
\text{ \ \ for all \ \ } y \in \R^d .
\end{eqnarray*}
Thus, \eqref{e:NtlV} also holds for function $V$, that is,
\begin{eqnarray}\label{e:NV}
|\nabla V(y)|
&\leq& C(1+|y|) \text{ \ \ for all \ \ } y \in \R^d,
\end{eqnarray}
and $\mathcal{A} V(y)=\mathcal{A}\tl{V}(y)$ for all $y \in \R^d$.
Combining with \eqref{e:AtlV}, the Lyapunov condition also holds for the function $V$, that is,
\begin{eqnarray}\label{e:AV}
\mathcal{A} V(y)
&\leq& -c_1 (\tl{V}(y)+\hat{c}_2)+\check{c}_1 +c_1\hat{c}_2
\ \leq \ -c_1 V(y)+\breve{c}_1
\text{ \ \ for any \ \ }y \in \R^d
\end{eqnarray}
with $\breve{c}_1=\check{c}_1 +c_1\hat{c}_2$.

\begin{lemma}\label{lem:AV2}
For $\mathcal{A}$ in \eqref{e:A}, $V$ in \eqref{e:Lypfun} and integers $\ell\geq 1$, there exists some positive constant $\breve{c}_{\ell}$ depending on $\ell$ such that
\begin{eqnarray*}
\mathcal{A} V^{\ell}(x)
&\leq& -c_1 V^{\ell}(x)+\breve{c}_{\ell},
\end{eqnarray*}
and
\begin{eqnarray}\label{e:Vm}
\mathbb{E} V^{\ell}(X_t^{x})
\ &\leq& \ e^{-c_1 t }V^{\ell}(x)+\frac{\breve{c}_{\ell}(1-e^{-c_1t})}{c_1},
\quad \forall t\geq 0.
\end{eqnarray}
In addition, $\mu(V^{\ell})\leq \frac{\breve{c}_{\ell}}{c_1}$ and $\mu(|\cdot|^{2\ell})\leq C$ where the constant $C$ depends on $\ell$. Furthermore, for any positive integers $\ell$ and probability measure $\nu$ satisfying $\nu(V^{\ell})<\infty$, there exist some positive constants $C$ and $c$, independent of $t$ such that
\begin{eqnarray*}
d_W(P_{t}^* \nu, \mu)  \ &\leq& \ C(1+\nu(V)) e^{-c t},  \\
\| P_{t}^* \nu - \mu\|_{{\rm TV}}  \ &\leq&  \ \| P_{t}^* \nu - \mu\|_{{\rm TV, V^{\ell}}}  \ \leq  \ C(1+\nu(V^{\ell})) e^{-c t}.
\end{eqnarray*}
\end{lemma}

In order to prove Proposition \ref{p:GeneralErgodicEM}, we first show that $(\tl{X}_k^{\eta})_{k\in \mathbb{N}_0}$ in \eqref{e:XD} is strong Feller and irreducible and the Lyapunov condition holds, and then we can get the exponential ergodicity from \cite{DFMS1, TT1}. We note that rather than using the classical Lyapunov function criterion, it is more convenient to use a modified version of the criterion in Douc et al. \cite{DFMS1} to establish the exponential ergodicity.

Rewrite \eqref{e:XD} as
\begin{eqnarray}\label{e:reXD}
\tl{X}_{k+1}^{\eta}
&=& \tl{X}_{k}^{\eta}+g(\tl{X}_{k}^{\eta}) \eta + \sigma (B_{(k+1){\eta}} -B_{k\eta}),
\end{eqnarray}
where $k\in \mathbb{N}_0$ and $\tl{X}^{\eta}_{0}$ is the initial value.

\begin{lemma}\label{lem:GePe}
For $\tl{\mathcal{P}}_{\eta}$ defined before and $f\in \mathcal{B}_b(\R^d,\R)$, one has
\begin{eqnarray*}
\| \nabla \tl{\mathcal{P}}_{\eta} f \|_{\infty} &\leq& \| f \|_{\infty} (1+C_{\rm op}\eta) \| \sigma^{-1} \|^2_{\rm op} \| \sigma \|_{\rm op}\eta^{-\frac{1}{2}}d^{\frac{1}{2}},
\end{eqnarray*}
which implies $(\tl{X}_k^{\eta})_{k\in \mathbb{N}_0}$ is strong Feller. Furthermore, $(\tl{X}_{k}^{\eta})_{k\in \mathbb{N}_0}$ is irreducible.
\end{lemma}

\begin{lemma}\label{lem:Xgesm}
For $(X^{x}_t)_{t\geq 0}$ and $(\tl{X}_{k}^{\eta,x})_{k\in \mathbb{N}_0}$ in \eqref{hSDEg} and \eqref{e:reXD}, respectively, and integers $\ell\geq 1$, there exists some positive constant $\tl{C}_{\ell}$ depending on $\ell$ not on $\eta$ such that for $0\leq s <1$,
\begin{eqnarray*}
\E |X^{x}_s-x|^{2\ell}  \ \leq \ \tl{C}_{\ell}(1+V^{\ell}(x))s^{\ell}
{\rm \ \ \ and \ \ \ }
\E |X^{x}_{\eta} - \tl{X}^{\eta,x}_{1}|^{2\ell}  \ \leq \  \tl{C}_{\ell} (1+V^{\ell}(x))\eta^{3\ell},
\end{eqnarray*}
where $V$ is in \eqref{e:Lypfun}.
\end{lemma}

\begin{lemma}\label{lem:PXeD}
For $(\tl{X}^{\eta,x}_{k})_{k\in \mathbb{N}_0}$ in \eqref{e:reXD}, $V$ in \eqref{e:Lypfun} and integers $\ell\geq 1$, there exist some constants $\check{\gamma}_{\ell} \in(0,1)$ and $\check{K}_{\ell} \in[0,\infty)$  such that
\begin{eqnarray*}
\tilde{\mathcal{P}}_{\eta} V^{\ell}(x) &\leq& \check{\gamma}_{\ell} V^{\ell}(x)+\check{K}_{\ell},
\end{eqnarray*}
where $\check{\gamma}_{\ell}=e^{-c_1 \eta} + \tl{C}_{\ell} \eta^{\frac{3}{2}}$ and $\check{K}_{\ell}= \frac{\breve{c}_{\ell}}{c_1}(1-e^{-c_1\eta}) + \tl{C}_{\ell} \eta^{\frac{3}{2}}$ with $c_1$ and $\breve{c}_{\ell}$ in Lemma \ref{lem:AV2} and the constant $ \tl{C}_{\ell} $ depends on $\ell$ but not on $\eta$.
\end{lemma}

We give proofs for Lemmas \ref{lem:AV2}, \ref{lem:GePe} \ref{lem:Xgesm} and \ref{lem:PXeD} at the end of Appendix \ref{App:GeneralErgodicEM}.

\begin{proof}[Proof of Proposition \ref{p:GeneralErgodicEM}]
The process $(\tl{X}_k^{\eta})_{k\in \mathbb{N}_0}$ is strong Feller and irreducible from Lemma \ref{lem:GePe}. Then $(\tl{X}_k^{\eta})_{k\in \mathbb{N}_0}$ has at most one invariant measure from Peszat and Zabczyk  \cite[Theorem 1.4]{PZ1}. Combining Feller property in Lemma \ref{lem:GePe} and Lyapunov condition in Lemma \ref{lem:PXeD}, we know $(\tl{X}_k^{\eta})_{k\in \mathbb{N}_0}$ is ergodic with unique invariant measure $\tl \mu_{\eta}$ from  Meyn and Tweedie  \cite[Theorem 4.5]{MT2}. Thus, we obtain the exponential ergodicity from \cite{DFMS1, TT1}.

For any $n\in \mathbb{N}_0$ and $x\in \R^d$, let $V$ be in \eqref{e:Lypfun},
\begin{eqnarray*}
V_{n}(x) \ = \ e^{\frac{c_1}{8} n\eta}(1+V(x)), \quad
r(n) \  = \  \frac{c_1}{8}\eta e^{\frac{c_1}{8}n\eta}, \quad
\Psi(x) \  =  \ 1+V(x),
\end{eqnarray*}
and
\begin{eqnarray*}
\mathscr{C}\  =  \ \{x: V(x) \  \leq  \ \frac{8 e^{\frac{c_1}{8}\eta}}{c_1 \eta}(1+\check{K}_{1}-\check{\gamma}_{1})-1 \},
\quad
b\  =  \ \frac{8e^{\frac{c_1}{8}\eta}} {c_1\eta}(1+\check{K}_{1}-\check{\gamma}_{1}),
\end{eqnarray*}
where $\check{\gamma}_{1}=e^{-c_1 \eta} + \tl{C}_{1} \eta^{\frac{3}{2}}$, $\check{K}_{1}= \frac{\breve{c}_{1}}{c_1}(1-e^{-c_1\eta}) + \tl{C}_{1} \eta^{\frac{3}{2}}$ with $c_1$ and $\breve{c}_{1}$ in \eqref{e:AV}. 
The set $\mathscr{C}$ is compact from \eqref{e:BV}. 
It follows from Lemma \ref{lem:PXeD} that
\begin{eqnarray*}
&& \tilde{\mathcal{P}}_{\eta} V_{n+1}(x) + r(n)\Psi(x) \\
&\leq&  \check{\gamma}_{1}e^{\frac{c_1}{8} \eta} V_n(x) +e^{\frac{c_1}{8}(n+1)\eta} (1+\check{K}_{1}-\check{\gamma}_{1})
+\frac{c_1}{8}\eta V_n(x) \\
&=& V_n(x)+\left( \check{\gamma}_{1}e^{\frac{c_1}{8}\eta} -1+\frac{c_1}{8}\eta \right) e^{\frac{c_1}{8}n\eta}(V(x)+1) +e^{\frac{c_1}{8}(n+1)\eta} (1+\check{K}_{1}-\check{\gamma}_{1})  \\
&=& V_n(x)+ \frac{c_1}{8} \eta e^{\frac{c_1}{8}n\eta}  \left( \frac{ \check{\gamma}_{1}e^{\frac{c_1}{8}\eta}-1 +\frac{c_1}{8}\eta}{ \frac{c_1}{8} \eta} (V(x)+1) + \frac{e^{\frac{c_1}{8}\eta}}{\frac{c_1}{8}\eta} (1+\check{K}_{1}-\check{\gamma}_{1}) \right) \\
&\leq& V_n(x) + b r(n)1_{\mathscr{C}}(x),
\end{eqnarray*}
where the last inequality holds from that
$\check{\gamma}_{1}e^{\frac{c_1}{8}\eta} -1+\frac{c_1}{8}\eta\leq -\frac{1}{8} c_1 \eta$ for small enough $\eta>0$.

We claim that the compact set $\mathscr{C}$ is petite.  It follows from Tuominen and Tweedie \cite[Theorem 2.1]{TT1} or Douc et al. \cite[Theorem 1.1]{DFMS1}  that
\begin{eqnarray*}
\lim_{n\to \infty} r(n) \|\tilde{\mathcal{P}}_{\eta}^n(x,\cdot) -\tl{\mu}_{\eta} \|_{\Psi} \ = \ 0,
\end{eqnarray*}
where $\|\tilde{\mathcal{P}}_{\eta}^n(x,\cdot) -\tl{\mu}_{\eta} \|_{\Psi}= \sup_{|h|\leq \Psi} | \tilde{\mathcal{P}}_{\eta}^n h(x)-\tl{\mu}_{\eta}(h)|$.

Since $h\in {\rm Lip}_0(1)$, it follows from the inequality \eqref{e:BV} that $h(x)\leq C(1+V(x))= C \Psi(x)$ for all $x\in \R^d$ with some constant $C\geq 1$. Then for any integers $k\geq 1$ and measure $\nu$ with $\nu(V)<\infty$, one has
\begin{eqnarray*}
d_W( (\tl{\mathcal{P}}_{\eta}^k)^* \nu, \tl{\mu}_{\eta})
\ &\leq& \ C\eta^{-1} e^{-c k\eta}, \\
\| (\tl{\mathcal{P}}_{\eta}^k)^* \nu- \tl{\mu}_{\eta} \|_{\rm TV}
\ &\leq& \ C\eta^{-1} e^{-c k\eta}.
\end{eqnarray*}
It follows from Lemma \ref{lem:PXeD} that
\begin{eqnarray*}
\int_{\R^d} \tilde{\mathcal{P}}_{\eta} V^{\ell}(x) \tl{\mu}_{\eta}(\dif x)
\ &\leq& \ \int_{\R^d} \check{\gamma}_{\ell} V^{\ell}(x) \tl{\mu}_{\eta}(\dif x) +\check{K}_{\ell},
\end{eqnarray*}
such that
\begin{eqnarray*}
\tl{\mu}_{\eta}(V^{\ell})
\ &\leq& \  \check{\gamma}_{\ell} \tl{\mu}_{\eta}(V^{\ell}) +\check{K}_{\ell},
\end{eqnarray*}
that is,
\begin{eqnarray*}
\tl{\mu}_{\eta}(V^{\ell})
\ &\leq& \ \frac{\check{K}_{\ell}}{1-\check{\gamma}_{\ell}}
\  =  \ \frac{\frac{\breve{c}_{\ell}}{c_1}(1-e^{-c_1\eta}) + \tl{C}_{\ell}  \eta^{\frac{3}{2}}}{1-e^{-c_1 \eta}- \tl{C}_{\ell}\eta^{\frac{3}{2}}}
\  \leq  \  \frac{2\frac{\breve{c}_{\ell}}{c_1} c_1 \eta }{\frac{1}{2}c_1 \eta}
\  = \  \frac{4\breve{c}_{\ell}}{c_1},
\end{eqnarray*}
where the last inequality holds from Taylor expansion for $e^{-c_1 \eta}$ with small $\eta>0$. This implies the desired inequality from the relationship between $V$ and $|\cdot|^2$ in \eqref{e:BV}.

To show the compact set $\mathscr{C}$ is petite.  It suffices to show that
\begin{eqnarray}\label{e:petite}
p(\eta,x,z) \ \geq \ c \nu(z), \quad \forall x\in \mathscr{C},
\end{eqnarray}
where $p(\eta,x,z)$ is the density of $\tl{X}_1^{\eta,x}$, $c$ is some positive constant and $\nu$ is a probability measure from Tuominen and Tweedie \cite[p. 778]{TT1}. Since
\begin{eqnarray}\label{e:pe}
p(\eta,x,z)
&=& ((2\pi)^d \eta^d {\rm det}(\sigma\sigma^{\prime}) )^{-\frac{1}{2}} \exp\left(-(z-x-\eta g(x))^{\prime} \frac{(\sigma\sigma^{\prime})^{-1}} {2\eta}(z-x-\eta g(x)) \right), \nonumber \\
\end{eqnarray}
denoting $\lambda_M$ and $\lambda_m$ as the maximum and minimum eigenvalues of matrix $\sigma \sigma^{\prime}$, respectively, and from the fact
$$
|z-x-\eta g(x)|^2 \ \leq \  2|z|^2 + 4|x|^2 + 8\tl{C}_{\rm op}^2 \eta^2(1+|x|^2), \quad  \forall x, z \in \R^d,
$$
we obtain that $p(\eta,x,z)$ is bigger than
\begin{eqnarray*}
\left( (2\pi)^d \eta^d (\frac{1}{2}\lambda_m)^d \right)^{-\frac{1}{2}} \exp\left(-\frac{|z|^2} {\lambda_m\eta}\right)
\left(\frac{2\lambda_M}{\lambda_m} \right)^{-\frac{d}{2}} \exp\left(-\frac{\lambda_m^{-1}} {2\eta}(4|x|^2 + 8\tl{C}_{\rm op}^2 \eta^2(1+|x|^2)) \right).
\end{eqnarray*}
Thus, inequality \eqref{e:petite} holds by taking
\begin{eqnarray*}
\nu(z) &=& \left((2\pi)^d \eta^d (\frac{1}{2}\lambda_m)^d \right)^{-\frac{1}{2}} \exp\left(-\frac{|z|^2} {\lambda_m\eta}\right),
\end{eqnarray*}
and
\begin{eqnarray*}
c&=& \inf_{ x \in \mathscr{C}}  \left\{ \left( \frac{2\lambda_M}{\lambda_m} \right)^{-\frac{d}{2}} \exp\left(-\frac{\lambda_m^{-1}} {2\eta}(4|x|^2 + 8\tl{C}_{\rm op}^2 \eta^2(1+|x|^2))\right) \right \}  >0
\end{eqnarray*}
for compact set $\mathscr{C}$. The proof is complete.
\end{proof}

\begin{proof}[Proof of Lemma \ref{lem:AV2}]
(i) Recall that $\mathcal{A} V(x)\leq -c_1V(x)+\breve{c}_1$ for all $x\in \R^d$  in \eqref{e:AV} and the function $V$ in \eqref{e:Lypfun}.
Combining \eqref{e:BV}, \eqref{e:NV} and using the Young's inequality, we obtain that there exist some positive constants $\breve{c}_{\ell}$ such that for all $x\in \R^d$ and integers $\ell \geq 1$, 
\begin{eqnarray*}
\mathcal{A}V^{\ell}(x)
&=& \ell V^{\ell-1}(x) \mathcal{A}V(x) + \frac{\ell(\ell-1)}{2}V^{\ell-2}(x) \langle \nabla V(x)(\nabla V(x))^{\prime}, \sigma \sigma^{\prime} \rangle_{\rm HS} \\
&\leq & -c_1 \ell V^{\ell}(x)+\breve{c}_1 \ell V^{\ell-1}(x) + \frac{\ell(\ell-1)}{2}V^{\ell-2}(x) \langle \nabla V(x)(\nabla V(x))^{\prime}, \sigma \sigma^{\prime} \rangle_{\rm HS} \\
&\leq&  -c_1 V^{\ell}(x) + \breve{c}_{\ell}.
\end{eqnarray*}
Using It\^{o}'s formula, we know for all $t\geq 0$, 
\begin{eqnarray*}
\mathbb{E} V^{\ell}(X_t^{x})
&=& V^{\ell}(x)+\int_0^t\mathbb{E} \mathcal{A}V^{\ell}(X_s^{x})  \dif s
\ \leq \  V^{\ell}(x)+\int_0^t ( -c_1\mathbb{E} V^{\ell}(X_s^{x})+\breve{c}_{\ell}) \dif s.
\end{eqnarray*}
This implies that (see Gurvich \cite[proof of Lemma 7.2]{Gur1})
\begin{eqnarray*}
\mathbb{E} V^{\ell}(X_t^{x})
&\leq& e^{-c_1 t }V^{\ell}(x)+\frac{\breve{c}_{\ell}(1-e^{-c_1t})}{c_1},
\quad \forall t\geq 0.
\end{eqnarray*}

Let $\chi: [0,\infty) \to [0,1]$ be a continuous function such that $\chi(r)=1$ for $0\leq r \leq 1$ and $\chi(r)=0$ for $r\geq 2$ and $L>0$ be a large number. It follows from \eqref{e:Vm} that
\begin{eqnarray*}
\mathbb{E} \left[ V^{\ell}(X_t^{x}) \chi\left(\frac{|X_t^{x}|}{L}\right) \right]
&\leq& e^{-c_1 t }V^{\ell}(x)+ \frac{\breve{c}_{\ell} (1-e^{-c_1 t})}{c_1}.
\end{eqnarray*}

First, taking $t\to \infty$, we have
\begin{eqnarray*}
\int_{\R^d} V^{\ell}(x) \chi\left(\frac{|x|}{L}\right) \mu(\dif x)
&\leq&  \frac{\breve{c}_{\ell}}{c_1},
\end{eqnarray*}
then as $L \to \infty$, one knows
\begin{eqnarray*}
\int_{\R^d} V^{\ell}(x) \mu(\dif x) &\leq&  \frac{\breve{c}_{\ell}}{c_1}.
\end{eqnarray*}
Combining with \eqref{e:BV}, we can get the inequality $\mu(|\cdot|^{2\ell})\leq C$ where $C$ depends on $\ell$.

(ii) With similar calculations from Dieker and Gao \cite[proof of Theorem 3]{DG1} and combining \eqref{e:Vm}, we obtain that there exist some positive constants $c$ and $C$, independent of $t$, such that for any positive integers $\ell$, 
\begin{eqnarray*}
\| P^*_{t}\nu-\mu \|_{\rm{TV}, \rm{V}^{\ell} }
&\leq& C (1+\nu(V^{\ell}))e^{-c t}.
\end{eqnarray*}
Thus, the exponential ergodicity for $(X_t)_{t\geq 0}$ in Wasserstein-1 distance holds from \eqref{e:dWandTV} and \eqref{e:BV} by taking $\ell=1$. Furthermore, one has the following inequality $\| P_{t}^* \nu-\mu \|_{ \rm{TV}} \leq   \| P_{t}^* \nu-\mu \|_{ \rm{TV}, \rm{V} }$ under $V(x)\geq 0$ for all $x\in \R^d$. The proof is complete.
\end{proof}

\begin{proof}[Proof of Lemma \ref{lem:GePe}]
(i) Since $\tl{X}_{1}^{\eta,x}$ has the same law as Gaussian distribution $\mathcal{N}(x+g(x)\eta, \eta\sigma\sigma^{\prime})$ with the density function $p(\eta,x,z)$ in \eqref{e:pe}, one has
\begin{eqnarray*}
\tilde{\mathcal{P}}_{\eta}f(x)
&=& \E f(\tl{X}_{1}^{\eta,x})
= \int_{\R^d} f(z)p(\eta, x, z) \dif z.
\end{eqnarray*}
Thus, 
\begin{eqnarray*}
\nabla \tilde{\mathcal{P}}_{\eta}f(x)
&=& \int_{\R^d} \eta^{-1} f(z)(I+\nabla g(x)\eta)(\sigma\sigma^{\prime})^{-1} (z-x-g(x)\eta) p(\eta,x,z) \dif z.
\end{eqnarray*}
It implies that
\begin{eqnarray*}
|\nabla \tilde{\mathcal{P}}_{\eta}f(x)|
&\leq& \int_{\R^d} \eta^{-1} |f(z)| \|I+\nabla g(x)\eta\|_{\rm op} \| \sigma^{-1} \|^2_{\rm op}|z-x-g(x)\eta| p(\eta,x,z) \dif z \\
&\leq& \| f \|_{\infty} (1+C_{\rm op}\eta) \| \sigma^{-1} \|^2_{\rm op} \| \sigma \|_{\rm op}\eta^{-\frac{1}{2}}d^{\frac{1}{2}},
\end{eqnarray*}
where the second inequality holds from that $\tl{X}_{1}^{\eta,x}-x-g(x)\eta$ has the same law as Gaussian distribution $\mathcal{N}(0,\eta\sigma\sigma^{\prime})$.

(ii) For any $x,y\in \R^d$ and $r>0$, it follows from \eqref{e:reXD} that
\begin{eqnarray*}
\tl{\mathcal{P}}_{\eta}(x,B(y,r))
&=& \PP(\tl{X}_{1}^{\eta,x} \in B(y,r) )
\ = \ \PP(x+g(x)\eta+\sigma B_{\eta} \in B(y,r)) \\
&=&\PP(\sigma B_{\eta} \in B(y-x- g(x)\eta,r))>0,
\end{eqnarray*}
where $B_{\eta}$ has the same law as Gaussian distribution $\mathcal{N}(0,\eta I)$. 

Assuming that  $\tilde{\mathcal{P}}_{\eta}^k(x,B(y,r))>0$ for any $x,y\in\R^d$, $r>0$ and some integer $k$, one has
\begin{eqnarray*}
\tilde{\mathcal{P}}_{\eta}^{k+1}(x,B(y,r))
=\int_{\R^d} \tilde{\mathcal{P}}_{\eta}^{k}(x,z) \tilde{\mathcal{P}}_{\eta}(z,B(y,r)) \dif z
>0.
\end{eqnarray*}
Thus, the irreducibility holds by induction.
\end{proof}

\begin{proof}[Proof of Lemma \ref{lem:Xgesm}]

(i) Combining with Lemma \ref{lem:AV2} and with similar calculations for \eqref{e:Vm}, one has
\begin{eqnarray}\label{e:V4}
\E V^{\ell}(X^{x}_{\eta})
&\leq& e^{-c_1 \eta} V^{\ell}(x) + \frac{\breve{c}_{\ell}}{c_1}(1-e^{-c_1 \eta}).
\end{eqnarray}

From $(X_t)_{t \geq 0}$ in SDE \eqref{hSDEg} and $0<s<1$, one has
\begin{eqnarray*}
X^{x}_s &=& x + \int_0^s  g(X^{x}_u) \dif u +\sigma B_s.
\end{eqnarray*}
Using the H\"{o}lder inequality and the estimate $|g(x)| \leq \tl{C}_{{\rm op}}(1+|x|)$ for all $x\in \R^d$ with $\tl{C}_{{\rm op}}$ in \eqref{tlCop}, one obtains that there exists some positive constant $\tl{C}_{\ell}$ depending on $\ell$ such that
\begin{eqnarray*}
|X^{x}_s-x|^{2\ell}
&=& \left| \int_0^s g(X^{x}_u)\dif u+\sigma B_s \right|^{2\ell}
\  \leq \   \tl{C}_{\ell} s^{2\ell-1} \int_0^s  (1+V^{\ell}(X^{x}_u) ) \dif u + \tl{C}_{\ell}| B_s |^{2\ell},
\end{eqnarray*}
where the last inequality holds from \eqref{e:BV}. Combining with \eqref{e:V4}, there exists some positive constant $\tl{C}_{\ell}$ depending on $\ell$ such that
\begin{eqnarray}\label{e:Exe}
\E |X^{x}_s -x |^{2\ell}
&\leq& \tl{C}_{\ell} s^{2\ell-1} \int_0^s  (1+\E V^{\ell}(X^{x}_u) ) \dif u + \tl{C}_{\ell} \E | B_s |^{2\ell}  \nonumber \\
& \leq &  \tl{C}_{\ell} s^{\ell}(1+V^{\ell}(x)),
\end{eqnarray}
where the second inequality holds from \eqref{e:V4} and the last inequality holds for $0<s<1$.

(ii) From  $(X_t)_{t\geq 0}$ and $(\tl{X}^{\eta}_{k})_{k\in \mathbb{N}_0}$ in \eqref{hSDEg} and \eqref{e:reXD}, respectively, one has
\begin{eqnarray*}
X^{x}_{\eta}  - \tl{X}^{\eta,x}_{1}
&=& \int_0^{\eta} ( g(X^{x}_s) - g(x) ) \dif s.
\end{eqnarray*}
Using the H\"{o}lder inequality and the estimate $\| \nabla g(x) \|_{{\rm op}}\leq C_{{\rm op}}$ for all $x\in \R^d$ with $C_{{\rm op}}$ in \eqref{Cop}, one obtains that there exists some positive constant $\tl{C}_{\ell}$ depending on $\ell$ not on $\eta$ such that
\begin{eqnarray*}
|X^{x}_{\eta}  - \tl{X}^{\eta,x}_{1} |^{2\ell}
&\leq& \left| \int_0^{\eta} ( g(X^{x}_s) - g(x) ) \dif s \right|^{2\ell}
\ \leq \ \tl{C}_{\ell} \eta^{2\ell-1} \int_0^{\eta} |X^{x}_s - x|^{2\ell} \dif s .
\end{eqnarray*}
Combining this with \eqref{e:Exe}, one obtains that there exists some positive constant $\tl{C}_{\ell}$ depending on $\ell$ not on $\eta$ such that
\begin{eqnarray*}
\E |X^{x}_{\eta} - \tl{X}^{\eta,x}_{1} |^{2\ell}
&\leq& \tl{C}_{\ell} \eta^{2\ell-1} \int_0^{\eta} s^{\ell}(1+V^{\ell}(x)) \dif s
\ \leq \ \tl{C}_{\ell} (1+V^{\ell}(x))\eta^{3\ell}.
\end{eqnarray*}
The proof is complete.
\end{proof}

\begin{proof}[Proof of Lemma \ref{lem:PXeD}]

We use the method from Mattingly et al. \cite[Theorem 7.2]{MSH1} to get the desired inequality. For $V$ in \eqref{e:Lypfun}, one has
\begin{eqnarray*}
\E V^{\ell}(\tl{X}^{\eta,x}_{1})
&\leq& \E V^{\ell}(X^{x}_{\eta})+\E| V^{\ell}(\tl{X}^{\eta,x}_{1})-V^{\ell}(X^{x}_{\eta})|.
\end{eqnarray*}
Let $\tl{C}_{\ell}$ be some constants depending on $\ell$ but not on $\eta$, whose values may vary from line to line.

We claim that for small $\eta \in (0,e^{-1})$,
\begin{eqnarray}\label{e:V-V4}
\E |V^{\ell}(\tl{X}^{\eta,x}_{1}) -V^{\ell}(X^{x}_{\eta})|
&\leq& \tl{C}_{\ell} (1+V^{\ell}(x))\eta^{\frac{3}{2}}.
\end{eqnarray}
Combining this with \eqref{e:V4}, one has
\begin{eqnarray*}
\E V^{\ell}(\tl{X}^{\eta,x}_{1})
&\leq& \E V^{\ell}(X^{x}_{\eta}) + \E | V^{\ell}(\tl{X}^{\eta,x}_{1}) -V^{\ell}(X^{x}_{\eta})| \\
&\leq& e^{-c_1 \eta} V^{\ell}(x) +  \frac{\breve{c}_{\ell}}{c_1}(1-e^{-c_1\eta}) + \tl{C}_{\ell} \eta^{\frac{3}{2}} (1+V^{\ell}(x))  \\
&\leq& (e^{-c_1 \eta} + \tl{C}_{\ell} \eta^{\frac{3}{2}} ) V^{\ell}(x) + \frac{\breve{c}_{\ell}}{c_1}(1-e^{-c_1\eta}) + \tl{C}_{\ell} \eta^{\frac{3}{2}},
\end{eqnarray*}
which implies
\begin{eqnarray*}
\tilde{\mathcal{P}}_{\eta} V^{\ell}(x)
&\leq& \check{\gamma}_{\ell} V^{\ell}(x)+\check{K}_{\ell},
\end{eqnarray*}
with $\check{\gamma}_{\ell}=e^{-c_1 \eta} + \tl{C}_{\ell} \eta^{\frac{3}{2}}$ and $\check{K}_{\ell}= \frac{\breve{c}_{\ell}}{c_1}(1-e^{-c_1\eta}) + \tl{C}_{\ell} \eta^{\frac{3}{2}}$ for small enough $\eta\in (0,e^{-1})$.

It remains to show the claim \eqref{e:V-V4} holds.
Since
\begin{eqnarray*}
&& V^{\ell}(\tl{X}^{\eta,x}_{1})-V^{\ell}(X^{x}_{\eta})  \\
&=& \ell \int_0^1 (\tl{X}^{\eta,x}_{1}-X^{x}_{\eta})^{\prime} \nabla V( X^{x}_{\eta} + r(\tl{X}^{\eta,x}_{1}-X^{x}_{\eta})  ) V^{\ell-1}( X^{x}_{\eta} + r(\tl{X}^{\eta,x}_{1}-X^{x}_{\eta})  )   \dif r,
\end{eqnarray*}
it follows from \eqref{e:BV} that
\begin{eqnarray*}
|V^{\ell}(\tl{X}^{\eta,x}_{1})-V^{\ell}(X^{x}_{\eta})|
&\leq& \tl{C}_{\ell} \int_0^1 [1+|X^{x}_{\eta} + r(\tl{X}^{\eta,x}_{1}-X^{x}_{\eta})|]^{2\ell-1}  |\tl{X}^{\eta,x}_{1}-X^{x}_{\eta}| \dif r  \\
& \leq & \tl{C}_{\ell} (1+|X^{x}_{\eta}|^{2\ell-1})| \tl{X}^{\eta,x}_{1}-X^{x}_{\eta}| +\tl{C}_{\ell} |\tl{X}^{\eta,x}_{1}-X^{x}_{\eta}|^{2\ell}.
\end{eqnarray*}
By the H\"{o}lder inequality and \eqref{e:BV}, one has
\begin{eqnarray*}
&& \E |V^{\ell}(\tl{X}^{\eta,x}_{1}) -V^{\ell}(X^{x}_{\eta})|  \\
&\leq& \tl{C}_{\ell}(\E [1+|X^{x}_{\eta}|^{2\ell-1}] ^{\frac{2\ell}{2\ell-1}})^{\frac{2\ell-1}{2\ell}} [\E |\tl{X}^{\eta,x}_{1}-X^{x}_{\eta}|^{2\ell}] ^{\frac{1}{2\ell}}+ \tl{C}_{\ell} \E |\tl{X}^{\eta,x}_{1}-X^{x}_{\eta}| ^{2\ell} \\
& \leq& \tl{C}_{\ell}(\E [1+V^{\ell}(X^{x}_{\eta})]) ^{\frac{2\ell-1}{2\ell}} [\E |\tl{X}^{\eta,x}_{1}-X^{x}_{\eta}|^{2\ell}] ^{\frac{1}{2\ell}}
+ \tl{C}_{\ell} \E |\tl{X}^{\eta,x}_{1}-X^{x}_{\eta}|^{2\ell}.
\end{eqnarray*}
Combining this with \eqref{e:V4} and Lemma \ref{lem:Xgesm},  for small $\eta \in (0,e^{-1})$, we have 
\begin{eqnarray*}
&& \E |V^{\ell}(\tl{X}^{\eta,x}_{1})-V^{\ell}(X^{x}_{\eta})|  \\
&\leq& \tl{C}_{\ell} \left(1+ e^{-c_1 \eta}V^{\ell}(x) +  \frac{\breve{c}_{\ell}}{c_1} \right)^{\frac{2\ell-1}{2\ell}} \eta^{\frac{3}{2}}(1+V^{\ell}(x))^{\frac{1}{2\ell}} + \tl{C}_{\ell} (1+V^{\ell}(x))\eta^{3\ell} \\
&\leq&\tl{C}_{\ell} (1+V^{\ell}(x)) \eta^{\frac{3}{2}}.
\end{eqnarray*}
The proof is complete.
\end{proof}

\section{The proofs for the lemmas and propositions in Section \ref{s:CLTMDP}} \label{sec:AAS}

\subsection{The proof of Proposition \ref{lem:regf}} 
\begin{proof}[Proof of Proposition \ref{lem:regf}]
(i) It follows from Eq. \eqref{e:SE1} and Lemma \ref{lem:AV2} that
	\begin{eqnarray}\label{e:SE2}
|f(x)| & \leq &  \int_0^{\infty} |P_t h(x) - \mu(h) |  \dif t
\ \leq \  \|h\|_{\infty} \int_0^{\infty} \| P^*_{t}\delta_x - \mu \|_{ \rm{TV}, \rm{V} }   \dif t  \nonumber \\
&\leq& C \|h\|_{\infty}(1+V(x)  )
\  \leq \   C \|h\|_{\infty}(1+|x|^2),
	\end{eqnarray}
where the last inequality holds from the relationship between $V$ and $|\cdot|^2$ in \eqref{e:BV}.

(ii) First, let  $h \in \mathcal{C}_b^1(\R^d,\R)$, we have $\nabla_u\E[h(X_t^x)]=\E[\nabla_u h(X_t^x)]$ by Lebesgue's dominated convergence theorem. Then we consider the term $ \nabla_u [e^{-\lambda t} \E f(X_t^x)]$.

	It follows from Lemma \ref{hLef2} that
	\begin{eqnarray*}
		\E [\nabla_u f(X_t^x)]
		&=& \E[\nabla f(X_t^x) \nabla_u X_t^x]
		\ = \ \E[\nabla f(X_t^x) D_{\mathbb{V}} X_t^x] \\
		&= & \E[D_{\mathbb{V}} f(X_t^x)]
	\ = \ \E[f(X_t^x)\mathcal{I}_{u}^x(t)].
	\end{eqnarray*}
From \eqref{e:SE2} and the estimate for $\E V^2(X_t^x)$ in \eqref{e:Vm} of Lemma \ref{lem:AV2}, one has
	\begin{eqnarray}\label{e:f2}
[ \E |f(X_t^x)|^2 ]^{\frac{1}{2}}
\ \leq \ C [ 1+ \E V^2(X_t^x)] ^{\frac{1}{2}}
\ \leq \ C e^{\frac{C_4}{2} t}(1+|x|^2).
	\end{eqnarray}
Combining this with the estimate for  $\E | \mathcal{I}_{u}^x(t)|^2$ in \eqref{e:IuexEst} and using the H\"{o}lder inequality, one has
	\begin{eqnarray*}
\E|f(X_t^x)\mathcal{I}_{u}^x(t)|
&\leq& [\E f^2 (X_t^x) ]^{\frac{1}{2}} [\E |\mathcal{I}_{u}^x(t) |^2 ]^{\frac{1}{2}}
\ \leq \ C e^{\frac{C_4}{2} t }(1+ |x|^2)  |u| t^{-\frac{1}{2}} e^{C_{{\rm op}}t}
\ < \  \infty.
	\end{eqnarray*}
With similar calculations, one has
	\begin{eqnarray*}
		\E|h(X_t^x)\mathcal{I}_{u}^x(t)|
		&\leq&  \| h \|_{\infty}  \E | \mathcal{I}_{u}^x(t) | 		
		\ < \  \infty.
	\end{eqnarray*}
	Here we fix $x$ and $t>0$. Then by Lebesgue's dominated convergence theorem, we have
	\begin{eqnarray*}
		\nabla_u\E[e^{-\lambda t}f(X_t^x)] \ =\ \E[e^{-\lambda t}\nabla_u f(X_t^x)].
	\end{eqnarray*}
	
	It follows from the H\"{o}lder inequality that
	\begin{eqnarray}\label{e:Nfx0}
	&& \int_0^{\infty} \left|e^{-\lambda t} (  \lambda  \nabla_{u}\E [ f (X_t^x)  ]  -  \nabla_u\E [   h(X_t^x)] )\right| \dif t \nonumber \\
	&=& \int_0^{\infty}\left| e^{-\lambda t} (  \lambda  \E [ f (X_t^x)  \mathcal{I}_{u}^x(t) ]  -  \E [ h(X_t^x) \mathcal{I}_u^x(t) ] ) \right|\dif t   \nonumber  \\
	 &\leq&  \int_0^{\infty} e^{-\lambda t}   \lambda  [ \E f^2 (X_t^x) ]^{\frac{1}{2}} [ \E | \mathcal{I}_{u}^x(t)  |^2 ]^{\frac{1}{2}}    \dif t
	+  \int_0^{\infty} e^{-\lambda t}   \| h \|_{\infty}  \E [ |\mathcal{I}_u^x(t)| ]  \dif t.
	\end{eqnarray}
	From estimates for  $\E | \mathcal{I}_{u}^x(t)  |^2$ in \eqref{e:IuexEst}  and  $[ \E |f(X_t^x)|^2 ]^{\frac{1}{2}} $ in \eqref{e:f2}, one has
\begin{eqnarray}\label{e:Nfx1}
\int_0^{\infty} e^{-\lambda t}   \lambda  [ \E f^2 (X_t^x) ]^{\frac{1}{2}} [ \E | \mathcal{I}_{u}^x(t)  |^2 ]^{\frac{1}{2}}  \dif t
&\leq& \int_0^{\infty}  \lambda (1+ |x|^2) e^{\frac{C_4}{2} t } |u| t^{-\frac{1}{2}} e^{ (-\lambda + C_{ {\rm op} })t  } \dif t \nonumber  \\
&\leq&  C (1+ |x|^2) |u|,
	\end{eqnarray}
	where the last inequality holds by taking $\lambda\geq \frac{C_4}{2} + C_{ { \rm op} }+1$.

	From estimate for  $\E | \mathcal{I}_{u}^x(t)  |$ in \eqref{e:IuexEst}, one has
	\begin{eqnarray}\label{e:Nfx2}
\int_0^{\infty} e^{-\lambda t}   \|  h \|_{\infty}  \E [ | \mathcal{I}_u^x(t) | ]    \dif t
&\leq& \int_0^{\infty} e^{-\lambda t}  \|  h \|_{\infty} \frac{C|u|}{ t^{1/2} }  e^{C_{ \textrm{op} } t} \dif t
\ \leq \ C\|h\|_{\infty} |u|,
	\end{eqnarray}
	where the last inequality holds by taking $\lambda \geq  C_{ \textrm{op} }+1$.

	Since
	\begin{eqnarray*}
		f(x)=  \int_0^{\infty} e^{-\lambda t} P_t[  \lambda f(x) - h(x) + \mu(h) ] \dif t,  \quad  \forall  \lambda>0,
	\end{eqnarray*}
	by Lebesgue's dominated convergence theorem, we have
	\begin{eqnarray}\label{e:nuf}
		\nabla_u f(x)&=&  \int_0^{\infty}  \nabla_u [\lambda e^{-\lambda t} \E f(X_t^x) - e^{-\lambda t} \E h(X_t^x)] \dif t \nonumber  \\
		&=& \int_0^{\infty} e^{-\lambda t} (  \lambda  \nabla_{u}\E [ f (X_t^x)  ]  -  \nabla_u\E [   h(X_t^x)] ) \dif t   \nonumber \\
		&=&\int_0^{\infty} e^{-\lambda t} (  \lambda  \E [ f (X_t^x)  \mathcal{I}_{u}^x(t) ]  -  \E [  h(X_t^x) \mathcal{I}_{u}^x(t) ] ) \dif t.
	\end{eqnarray}	
	Taking $\lambda= \frac{C_4}{2} + C_{ \textrm{op} } +1$, combining \eqref{e:Nfx0}, \eqref{e:Nfx1} and \eqref{e:Nfx2}, one has
	\begin{eqnarray*}
|\nabla_{u} f(x)|
&\leq& C \| h \|_{\infty} (1+ |x|^2) |u|.
	\end{eqnarray*}

Second, we extend $h\in \mathcal{C}_b^1(\R^d,\R)$ to $h\in \mathcal{B}_b(\R^d,\R)$ by the standard approximation (see Fang et al. \cite[pp. 968-969]{FSX1}). Let $h\in \mathcal{B}_b(\R^d,\R)$, and define
	\begin{eqnarray*}
		h_{\delta}(x) = \int_{\R^d} \varphi_{\delta}(y) h(x-y) \dif y, \quad \delta>0,
	\end{eqnarray*}
where $\varphi_{\delta}$ is the density function of the normal distribution $\mathcal{N}(0,\delta^2 I)$. Thus $h_{\delta}$ is smooth, $\| h_{\delta} \|_{\infty}  \leq \| h\|_{\infty}$ and the solution to the Poisson equation \eqref{e:SE} with $h$ replaced by $h_{\delta}$, is
         \begin{eqnarray*}
	f_{\delta}(x)= - \int_0^{\infty} P_t[   h_{\delta}(x) - \mu(h_{\delta}) ] \dif t.
	\end{eqnarray*}

	Denote $\hat{h}_{\delta} = -h_{\delta} + \mu(h_{\delta})$.  Since $h_{\delta} \in \mathcal{C}_b^1(\R^d,\R)$, one has $h_{\delta}(x) \leq \|h_{\delta}\|_{\infty}( 1+ V(x))$ for all $x\in \R^d$. From Lemma \ref{lem:AV2}, one has
	\begin{eqnarray*}
|P_t h_{\delta}(x) - \mu(h_{\delta}) |
&\leq& \|h_{\delta}\|_{\infty} \| P^*_{t}\delta_x - \mu \|_{\rm{TV},\rm{V}}
\ \leq \ C \|h_{\delta}\|_{\infty}(1+V(x)) e^{-ct}.
	\end{eqnarray*}
With similar calculations in the proof for Lemma \ref{prop:ST}, one has
	\begin{eqnarray*}
|f_{\delta}(x)|
&\leq&  C \|h_{\delta}\|_{\infty}(1+ |x|^2 ).
\end{eqnarray*}
By the dominated convergence theorem, one has
	\begin{eqnarray*}
	\lim_{\delta \to 0} f_{\delta}(x) \ = \ - \int_0^{\infty} P_t[   h(x) - \mu(h) ] \dif t
	\ = \ f(x).
	\end{eqnarray*}
From \eqref{e:SE2} and $\| h_{\delta} \|_{\infty} \leq \| h \|_{\infty}$, we know
	\begin{eqnarray*}
| f_{\delta}(x)|
&\leq& C \|h_{\delta}\|_{\infty}( 1 + |x|^2 )
\ \leq \ C \|h\|_{\infty}(1+ |x|^2).
\end{eqnarray*}
Let $\delta \to 0$, we know
\begin{eqnarray*}
| f(x)|
\ \leq \ C \|h\|_{\infty}( 1 + |x|^2 ).
\end{eqnarray*}

 From calculations for \eqref{e:nuf}, one has
	\begin{eqnarray*}
\nabla_u f_{\delta}(x)
&=&\int_0^{\infty} e^{-\lambda t} (  \lambda  \E [ f_{\delta} (X_t^x)  \mathcal{I}_{u}^x(t) ]  -  \E [  h_{\delta}(X_t^x) \mathcal{I}_{u}^x(t) ] ) \dif t.
	\end{eqnarray*}
With similar calculations for the proof of \eqref{e:SE2}, one has
\begin{eqnarray*}
|\nabla_{u} f_{\delta}(x)|
&\leq&  C \| h_{\delta} \|_{\infty}  (1+ |x|^2) |u|
\ \leq \ C \| h \|_{\infty}(1+|x|^2) |u|.
\end{eqnarray*}
	Since the operator $\nabla$ is closed (see Partington \cite[Theorem 2.2.6]{PJR1}), it follows from the dominated convergence theorem that
	\begin{eqnarray*}
	         \lim_{\delta \to 0} \nabla_u f_{\delta}(x)
	         &=& \lim_{\delta \to 0}  \int_0^{\infty} e^{-\lambda t} (  \lambda  \E [ f_{\delta} (X_t^x)  \mathcal{I}_{u}^x(t) ]  -  \E [  h_{\delta}(X_t^x) \mathcal{I}_{u}^x(t) ] ) \dif t \\
	         &=&  \int_0^{\infty} e^{-\lambda t} (  \lambda  \E [ f (X_t^x)  \mathcal{I}_{u}^x(t) ]  -  \E [  h(X_t^x) \mathcal{I}_{u}^x(t) ] ) \dif t \\
	         &=& \nabla_u f(x).
	\end{eqnarray*}
	 Letting $\delta \to 0$, we obtain
     \begin{eqnarray*}
|\nabla_{u} f(x)|
&\leq&  C \| h \|_{\infty}  (1+ |x|^2) |u|.
	\end{eqnarray*}
The proof is complete.
\end{proof}

\begin{proof}[Proof of Lemma \ref{prop:ST}]
We first show that $\int_0^{\infty}[P_t h(x) - \mu(h)] \dif t$ is well defined. Denote $\hat{h} = -h + \mu(h)$. For  any $h \in \mathcal{B}_b(\R^d,\R)$, we have that  $h(x) \leq \|h\|_{\infty}( 1+ V(x))$ for all $x\in \R^d$, and from Lemma \ref{lem:AV2} that
	\begin{eqnarray*}
	|P_t h(x) - \mu(h) | \ \leq \ \|h\|_{\infty} \| P^*_{t}\delta_x - \mu    \|_{ \rm{TV}, \rm{V} }
	\ \leq \ C \|h\|_{\infty}(1+ V(x)) e^{-c t}.
	\end{eqnarray*}
This implies that
	\begin{eqnarray*}
\left| \int_0^{\infty}[P_t h(x)-\mu(h)] \dif t \right|
\ \leq \
C \|h\|_{\infty}(1+V(x))
\ < \ \infty.
	\end{eqnarray*}
The reminder is similar to that of {Fang et al. \cite[Proposition 6.1]{FSX1}.} The proof is complete.
\end{proof}

\medskip 

\subsection{Properties of the mollified diffusion}

\begin{proof}[Proof of Lemma \ref{lem:XXem}]
(i) Using It\^o's formula, for any integers $m \geq 2$, we have
\begin{eqnarray*}
\E |X^{x}_{t}|^{m}&=&|x|^{m}+m \mathbb{E} \int_{0}^{t} |X^{x}_{s}|^{m-2} (X^{x}_{s})' g(X^{x}_{s}) \dif s \\
&&+\frac m2 \mathbb{E} \int_{0}^{t} |X^{x}_{s}|^{m-4}\left[(m-2)|\sigma' X^{x}_{s}|^{2}+ { \rm tr} (\sigma \sigma') |X^{x}_{s}|^{2} \right] \dif s.
\end{eqnarray*}
By the bound $|g(x)| \leq \tl{C}_{\rm op}(1+|x|)$ for all $x\in \R^d$ with $\tl{C}_{{\rm op}}$ in \eqref{tlCop}, we further get
\begin{eqnarray*}
\E |X^{x}_{t}|^{m}& \leq &|x|^{m}+\tilde{C}_{\rm op} \left(  m\int_{0}^{t} \E |X^{x}_{s}|^{m} \dif s+  m\int_{0}^{t} \E |X^{x}_{s}|^{m-1} \dif s+ m^2\int_{0}^{t} \E |X^{x}_{s}|^{m-2} \dif s\right)  \nonumber \\
&\le & |x|^{m}+2m^2 \tilde{C}_{\rm op} \left(  \int_{0}^{t} \E |X^{x}_{s}|^{m} \dif s+t \right),
\end{eqnarray*}
where the second inequality is by the Young's inequality. Thus, we have
\begin{eqnarray*}
\E |X^{x}_t|^{m}  \ \le \  e^{C_m t}(|x|^{m}+1),
\end{eqnarray*}
where $C_m=2m^2 \tl{C}_{\rm op}$. Since $|g_\e(x)| \leq  \tl{C}_{\rm op}(1+|x|)$ uniformly for $\e \in (0,1)$, the moment estimates $\E |X^{\e,x}_t|^{m}$ can be obtained similarly. 
Thus, inequality \eqref{e:XXem} holds.

(ii) Consider $X^{\e,x}_{t}-X_{t}^{x}$, which satisfies the following equation
\begin{eqnarray*}
\frac{\dif}{\dif t} \left(X^{\e,x}_{t}-X_{t}^{x}\right)&=&g_\e (X^{\e,x}_t)-g(X^{x}_t) \\
&=&g_\e (X^{\e,x}_t)-g_{\e}(X^{x}_t)+g_{\e}(X^{x}_t)-g(X^{x}_{t}), \\
&=&\nabla g_{\e}(\theta_{t})  \left(X^{\e,x}_{t}-X_{t}^{x}\right)+g_{\e}(X^{x}_t)-g(X^{x}_{t}),
\end{eqnarray*}
where $\theta_{t}$ is between $X^{x}_t$ and $X^{\e,x}_t$. The above equation can be solved by
\begin{eqnarray*}
X^{\e,x}_{t}-X_{t}^{x}&=&\int_{0}^{t} \exp\left(\int_{s}^{t} \nabla g_{\e}(\theta_{r}) \dif r \right) (g_{\e}(X^{x}_s)-g(X^{x}_{s})) \dif s.
\end{eqnarray*}
Since $\|\nabla g_{\e}(x)\|_{ {\rm op} } \le C_{{\rm op}}$ for all $x \in \R^d$, the  relation $|g_{\e}(x)-g(x)| \le C_{{\rm op}}\e$ for all $x \in \R^d$ immediately gives us \eqref{e:XeCon-1}.
\end{proof}

\begin{proof}[Proof of Lemma \ref{l:XeCon}]
Denote the event
$$N=\left\{\int_{0}^{\infty} \|\nabla g_{\e}(X^{\e,x}_{s})\|_{ {\rm op} }1_{\{{\rm e}^{\prime}  X^{x}_{s}=0\}} \dif s\ne 0\right\}.
$$
We claim that
\Be  \label{e:PN=0}
\PP(N) \ = \ 0.
\Ee
Indeed, for any $T>0$, by $\|\nabla g_{\e}(x)\|_{{\rm op}} \le C_{{\rm op}}$ for all $x \in \R^d$ and $\e$,  we have
\begin{eqnarray}\label{e:ET}
\E \int_{0}^{T} \|\nabla g_{\e}(X^{\e,x}_{s})\|_{ {\rm op} }1_{\{{\rm e}^{\prime}  X_{s}^x=0\}} \dif s
&=& \int_{0}^{T} \E[\|\nabla g_{\e}(X^{\e,x}_{s})\|_{{\rm op}}1_{\{{\rm e}^{\prime}  X_{s}^x=0\}}] \dif s  \nonumber \\
& \leq &   \int_{0}^{T}  C_{{\rm op}}  \E 1_{\{{\rm e}^{\prime}  X_{s}^x=0\}} \dif s  \nonumber \\
&=& 0,
\end{eqnarray}
where the last equality is by Proposition \ref{lem:occupation} and the fact that for any small $\eps>0$,
\begin{eqnarray*}
 \E\int_0^T  1_{\{{\rm e}^{\prime}  X_{s}^x=0\}} \dif s
& \leq &  \E L_T^{\eps,x}
 \ \leq \ C\eps e^{\frac{C_2}{2} T} (1+|x|)(1+T),
\end{eqnarray*}
while $L_t^{\eps,x} = \int_0^t [-\frac{1}{\eps^2} ({\rm e}^{\prime} X_s^x)^2 +1 ] 1_{ \{ |{\rm e}^{\prime} X_s^x| \leq \eps \} } \dif s$. The above inequality holds for any $\eps>0$, we know $  \E\int_0^T  1_{\{{\rm e}^{\prime}  X_{s}^x=0\}} \dif s = 0$.

Since \eqref{e:ET} holds for all $T>0$, we see that
\begin{eqnarray*}
\E \int_{0}^{\infty} \|\nabla g_{\e}(X^{\e,x}_{s})\|_{ {\rm op} }1_{\{{\rm e}^{\prime}  X^{x}_{s}=0\}} \dif s
& = & 0,
\end{eqnarray*}
hence \eqref{e:PN=0} holds.

Recall the definition of $J^{\e,x}_{s,t}$ and define
$$\hat J^{\e,x}_{s,t}:=\exp \left(\int_s^t \nabla g_{\e} (X_{r}^{\e,x}) 1_{\{{\rm e}^{\prime}  X^{x}_{r} \ne 0\}}\dif r \right).
$$
It is easy to verify that
\Be  \label{e:JstConN}
\lim_{\e \rightarrow 0} \hat J^{\e,x}_{s,t}=J^{x}_{s,t}, \ \ \ \ \ \ 0 \le s \leq t<\infty.
\Ee
For any $\omega \notin N$, we know $\int_{0}^{\infty} \nabla g_{\e}(X^{\e,x}_{s})1_{\{{\rm e}^{\prime}  X^{x}_{s}=0\}} \dif s=0$ and thus
$$\exp\left(\int_{s}^{t} \nabla g_{\e}(X^{\e,x}_{r})1_{\{{\rm e}^{\prime}  X^{x}_{r}=0\}} \dif r \right)=I, \ \ \ \ 0 \le s \leq t<\infty.$$
Since $I$ commutes with any matrix, for all $\omega \notin N$, we get
$$\hat J^{\e,x}_{s,t}=\hat J^{\e,x}_{s,t} \exp \left(\int_{s}^{t} \nabla g_{\e}(X^{\e,x}_{r})1_{\{{\rm e}^{\prime}  X^{x}_{r}=0\}} \dif r \right)=J^{\e,x}_{s,t}, \ \ \ \  \ 0 \le s \leq t<\infty.$$
This, combining with \eqref{e:JstConN}, implies that for all $\omega \notin N$,
$$\lim_{\e \rightarrow 0}  J^{\e,x}_{s,t}=J^{x}_{s,t}, \ \ \ \ \ \ 0 \le s \leq t<\infty.$$
Note that $ J^{\e,x}_{s,t}$ and $J^{x}_{s,t}$ are matrices, hence the above pointwise convergence implies the convergence in operator.
\end{proof}

\subsection{Bismut's formula}

\begin{proof}[Proof of Lemma \ref{lem:EIm}]
Using the Burkholder-Davis-Gundy inequality, we have
\begin{eqnarray*}
\E \left|\mathcal{I}_{u}^{\e,x}(t)\right|^m
&\leq &\frac{C}{t^m}\E \left(\int_0^t  |\sigma^{-1}J^{\e,x}_{r} u|^2 \dif r \right)^{m/2}
\le  \frac{ C }{t^m} \E \left(\int_0^t \|\sigma^{-1}\|^{2}_{ {\rm op} } \|J^{\e,x}_{r}\|^{2}_{ {\rm op} }  |u|^{2}\dif r \right)^{m/2},
\end{eqnarray*}
which, together with \eqref{e:JJe}, immediately gives \eqref{e:IuexEst}.

For the second relation, by Burkholder-Davis-Gundy inequality, we have
\begin{eqnarray*}
\E |\mathcal{I}_{u}^{\e,x}(t)- \mathcal{I}_{u}^{x}(t)|^m
&=& \E \left|\frac1t\int_0^t \langle \sigma^{-1}(J^{\e,x}_{r}-J^{x}_{r}) u, \dif B_r\rangle \right|^m \\
&\leq &\frac{ C }{t^m} \E \left(\int_0^t  |\sigma^{-1}(J^{\e,x}_{r}-J^{x}_{r}) u|^2 \dif r \right)^{m/2} \ \rightarrow \ 0  { \ \ \rm as \ \ } \e \rightarrow 0,
\end{eqnarray*}
where the limit is by dominated convergence theorem (with a notice of Lemma \ref{l:XeCon}).
\end{proof}

\medskip 
\begin{proof}[Proof of Lemma \ref{hLef2}]
If $\psi\in \mathcal{C}^1(\mathbb{R}^d,\mathbb{R})$, then by \eqref{e:DVNu} and the Bismut's formula \eqref{e:BisFor}, we have
\begin{eqnarray*}
\nabla_{u} \mathbb{E}[\psi(X^{\e,x}_t)] &=& \mathbb{E}[\nabla \psi(X^{\e,x}_t) \nabla_{u} X^{\e,x}_t  ]
\ \ = \ \  \mathbb{E}[\nabla \psi(X^{\e,x}_t) D_{\mathbb{V}}X^{\e,x}_t] \\
&=& \mathbb{E}[D_{\mathbb{V}} \psi(X^{\e,x}_t)]
\ \ = \ \ \mathbb{E} [\psi(X^{\e,x}_t) \mathcal{I}_{u}^{\e,x}(t)],
\end{eqnarray*}
where $\mathbb{V}$ is the direction of Malliavin derivative.

Since the operator $\nabla$ is closed and by the well known property of closed operators (see Partington \cite[Proposition 2.1.4]{PJR1}), as long as it is shown that
\begin{eqnarray}  \label{e:EPhi}
&& \lim_{\e \rightarrow 0} \mathbb{E}[\psi(X^{\e,x}_t)] \ = \ \mathbb{E}[\psi(X^{x}_t)],
\  \ \lim_{\e \rightarrow 0} \mathbb{E} [\psi(X^{\e,x}_t) \mathcal{I}_{u}^{\e,x}(t)]\ =\ \mathbb{E} [\psi(X^{x}_t) \mathcal{I}_{u}^{x}(t)],
\end{eqnarray}
then we know that $\nabla_u\E[\psi(X^x_t)]$ exists and has its value as $\mathbb{E} [\psi(X^{x}_t) \mathcal{I}_{u}^{x}(t)]$. Hence, the first relation is proved.

Before proving \eqref{e:EPhi}, let us show that for all $m \ge 1$,
\begin{eqnarray}  \label{e:PhiPowM}
\mathbb{E} |\psi(X^{\e,x}_t)|^{m},  \mathbb{E} |\psi(X^{x}_t)|^{m} & \le & \hat{C}_1 e^{C_m t} (|x|^{m}+1),
\end{eqnarray}
where $\hat{C}_1$ does not depend on $\e$, $t$ and $x$. (Without loss of generality, we can take $\hat{C}_1$ depending on $m$, $\| \nabla \psi \|_{\infty}$ and $\psi(0)$.) Indeed, it is easily seen that
 \begin{eqnarray*}
\mathbb{E} |\psi(X^{\e,x}_t)|^{m}
& \le & \hat{C}_2 \left(|\psi(0)|^{m}+ \| \nabla \psi \|^{m}_{\infty} \E |X^{\e,x}_t|^{m} \right),
 \end{eqnarray*}
where $\hat{C}_2$ only depends on $m$ and this, together with \eqref{e:XXem}, immediately yields the aimed inequality.

For the first limit in \eqref{e:EPhi}, for all $m \ge 1$, by \eqref{e:XeCon-1}  we have
 \begin{eqnarray}
 \mathbb{E} |\psi(X^{\e,x}_t)-\psi(X^{x}_{t})|^{m} & \le & \| \nabla \psi \|^{m}_{\infty}  \E |X^{\e,x}_t-X^{x}_t |^{m}  \ \ \rightarrow \ 0 \textrm{ \ \ \rm as \ \ } \e \to 0.   \label{e:PhiXXe}
 \end{eqnarray}
 For the second one, we have
 \begin{eqnarray*}
&& |\mathbb{E} [\psi(X^{\e,x}_t) \mathcal{I}_{u}^{\e,x}(t)]-\mathbb{E} [\psi(X^{x}_t) \mathcal{I}_{u}^{x}(t)]|  \\
 &\leq &  \mathbb{E} [|\psi(X^{\e,x}_t)| |\mathcal{I}_{u}^{\e,x}(t)-\mathcal{I}_{u}^{x}(t)|]+\mathbb{E} [|\psi(X^{\e,x}_t)-\psi(X^{x}_t)||\mathcal{I}_{u}^{x}(t)|] \\
 & \rightarrow & 0 { \ \ \rm as  \ \ } \e \rightarrow 0,
 \end{eqnarray*}
where the convergence is by the following argument: applying the Cauchy-Schwarz inequality on the two expectations, and using \eqref{e:IuexEst}, \eqref{e:IueCon}, \eqref{e:PhiPowM} and \eqref{e:PhiXXe}.
\end{proof}

\subsection{Weighted occupation time}
\begin{proof}[Proof of Proposition \ref{lem:occupation}]
Denote
\begin{eqnarray}\label{e:rhobar0}
\phi_{\eps}(y)=
\left\{
\begin{array}{lllll}
\frac{2}{3}\eps y -\frac{1}{4}\eps^2,& \text{if } y>\eps,\\
-\frac{1}{12\eps^2}y^4 + \frac{1}{2}y^2,  & \text{if }   -\eps \leq y\leq \eps, \\
-\frac{2}{3}\eps y-\frac{1}{4}\eps^2, & \text{if } y <-\eps.
\end{array}
\right.
\end{eqnarray}
It is easy to check that
$$\ddot{\phi}_{\eps}(y) \ = \ \left[ \frac{-1}{\eps^2}y^2+1 \right]1_{ \{ |y|\leq \eps \}  },$$
$|\phi_{\eps}(y)|\leq C\eps |y|+C\eps^2$, $|\dot{\phi}_{\eps}(y)|\leq C\eps$ and $|\ddot{\phi}_{\eps}(y)|\leq 1$ for all $y\in \R$ and the positive constant $C$ is independent of $\eps$.

Applying It\^{o}'s formula to the function $\phi_{\eps}$, one has
\begin{eqnarray*}
\phi_{\eps}( {\rm e^{\prime}} X_t^x)
&=& \phi_{\eps}( {\rm e^{\prime} } x)
+\int_0^t \dot{\phi}_{\eps}({\rm e^{\prime}} X_s^x) {\rm e^{\prime}}g(X_s^x) \dif s + \int_0^t \dot{\phi}_{\eps}( {\rm e^{\prime}}X_s^x ) {\rm e^{\prime}}\sigma \dif B_s + \frac{|\sigma^{\prime} {\rm e}|^2}{2} L_t^{\eps,x},
\end{eqnarray*}
which implies that
\begin{eqnarray*}
\E L_t^{\eps,x}  &\leq& C \left[ \E |\phi_{\eps}( {\rm e^{\prime}} X_t^x)|
+|\phi_{\eps}( {\rm e^{\prime} } x)|
+\E \left|\int_0^t \dot{\phi}_{\eps}({\rm e^{\prime}} X_s^x) {\rm e^{\prime}}g(X_s^x) \dif s \right|  \right]   \\
&\leq& C\eps e^{\frac{C_2}{2}t}(1+|x|)(1+t),
\end{eqnarray*}
where the last inequality holds from Lemma \ref{lem:XXem}. The proof is complete.
\end{proof}
\end{appendix}

\section*{Acknowledgments.}
L. Xu is supported in part by NSFC grant (No. 12071499), Macao S.A.R grant FDCT  0090/2019/A2 and University of Macau grant  MYRG2018-00133-FST.
G. Pang is supported in part by  the US National Science Foundation grants DMS-1715875 and DMS-2216765. 
X. Jin was supported in part by the Fundamental Research Funds for the Central Universities grant (JZ2022HGQA0148).






%

\end{document}